\documentclass{article}
\usepackage{amssymb}
\usepackage{amsfonts}
\usepackage{amsmath}
\usepackage{tikz-cd}
\usepackage{faktor}
\usepackage{graphicx,nicefrac}
\usepackage[nottoc]{tocbibind}
\usepackage{microtype}
\usepackage{capt-of}
\usepackage{hyperref}

\newtheorem{theorem}{Theorem}[section]

\newtheorem{corollary}[theorem]{Corollary}

\newtheorem{definition}[theorem]{Definition}
\newtheorem{example}[theorem]{Example}

\newtheorem{lemma}[theorem]{Lemma}

\newtheorem{remark}[theorem]{Remark}

\renewcommand{\theequation}{\thesection.\arabic{equation}}
\newenvironment{proof}[1][Proof]{\noindent\textbf{#1.} }{\ \rule{0.5em}{0.5em}}

\input{tcilatex}
\begin{document}

\title{Algebraic structures on parallelizable manifolds}
\author{Sergey Grigorian \\
School of Mathematical \& Statistical Sciences\\
University of Texas Rio Grande Valley\\
Edinburg, TX 78539\\
USA}
\maketitle

\begin{abstract}
In this paper we explore algebraic and geometric structures that arise on
parallelizable manifolds. Given a parallelizable manifold $\mathbb{L}$,
there exists a global trivialization of the tangent bundle, which defines a
map $\rho_p:\mathfrak{l} \longrightarrow T_p \mathbb{L}$ for each point $p
\in \mathbb{L}$, where $\mathfrak{l}$ is some vector space. This allows us
to define a particular class of vector fields, known as fundamental vector
fields, that correspond to each element of $\mathfrak{l}$. Furthermore,
flows of these vector fields give rise to a product between elements of $%
\mathfrak{l}$ and $\mathbb{L}$, which in turn induces a local loop structure
(i.e. a non-associative analog of a group). Furthermore, we also define a
generalization of a Lie algebra structure on $\mathfrak{l}$. We will
describe the properties and examples of these constructions.
\end{abstract}

\tableofcontents

\section{Introduction}

\setcounter{equation}{0} The tangent bundle is a fundamental aspect of
smooth manifolds, and its global trivialization, or absolute parallelism,
has been a key area of study in differential geometry since the field's
emergence in the early 20th century \cite%
{cartan1926riem,EisenhartParallelism}. The study of parallelizable
manifolds, as smooth manifolds with a trivial tangent bundle are now
generally referred to, has been continuing since then in several directions.
One approach, which is more topological in nature, reframed the question in
terms of \emph{stably }parallelizable manifolds, i.e. manifolds for which
the \emph{stable }tangent bundle is trivial. Given a real smooth manifold $M$
with tangent bundle $TM$, and a trivial rank $1$ bundle $\varepsilon $ over $%
M$ (i.e. $M\times \mathbb{R}$), the stable tangent bundle is $TM\oplus
\varepsilon $. The conditions for a manifold to be stably parallelizable
(also known as a $\pi $-manifold) in low dimensions can be expressed in
terms of vanishing of certain characteristic classes \cite{HatcherVB}. It
also follows from \cite{BredonPimanifolds} that in dimensions $1,3,$ and $7,$
manifolds are parallelizable if and only if they are stably parallelizable.
Another approach, which stems from the origins of Riemannian geometry,
involves the study of flat metric connections on Riemannian manifolds.
Indeed, the existence of a flat metric connection allows to parallel
transport a frame from a single point to the entire manifold, and thus
obtain a trivialization of the tangent bundle. The converse is also
trivially true. Generally such connections will admit torsion and then
questions about classification of parallelizable manifolds become related to
the properties of torsion of flat metric connections. In \cite%
{cartan1926riem}, Cartan and Schouten have shown initiated this approach by
showing that a Riemannian manifold with a flat metric connection and totally
skew-symmetric torsion is either a compact simple Lie group or the $7$%
-sphere $S^{7},$ depending on whether the torsion is parallel or not. A
similar classification in the pseudo-Riemannian case has been obtained by
Wolf \cite{WolfParallelism1,WolfParallelism2}, again under the assumption of
totally skew-symmetric torsion. A more modern treatment of these approaches
has been given in \cite{AgricolaFriedrichFlat}. However, surprisingly,
despite significant efforts from different points of view, there is still no
full classification of parallelizable manifolds and shows that there is
still a need to explore properties of parallelizable manifolds.

This paper builds upon these foundational concepts to explore new algebraic
structures on parallelizable manifold. These structures extend the
traditional notions of Lie algebras and Lie groups, providing a broader
perspective on the geometric and algebraic properties of these manifolds.
Indeed, Lie groups form one of the most well-known classes of parallelizable
manifolds and are of course characterized by a globally defined smooth
associative product and a Lie algebra structure on the tangent space at
identity. The associativity property may be relaxed to consider \emph{smooth
loops }\cite{GrigorianLoops}. The $7$-sphere from \cite{cartan1926riem} is
an example of a smooth loop, when regarded as the set of unit octonions, and
is of course parallelizable. The smooth product allows to identify any
tangent space with the tangent space at identity, providing a global
trivialization of the tangent bundle. This also allows to define an bracket
algebra on the tangent space at identity, however, this is in general no
longer a Lie algebra. As shown in \cite{GrigorianLoops}, on a smooth loop we
may define a family of brackets, defined for each point of the loop, leading
to a \emph{bracket function} defined on the loop. The non-constant nature of
this function then leads to a non-trivial right hand side in the Jacobi
identity. On Lie group, the bracket of left- or right-invariant vector
fields is itself left- or right-invariant, respectively, and hence only a
unique bracket is defined.

In this paper, we expand some of the findings from \cite{GrigorianLoops},
but for arbitrary parallelizable manifolds. Indeed, several properties of
Lie groups and smooth loops only depend on the trivialization of the tangent
bundle, which is a weaker property than the existence of a global structure.

In Section \ref{sectPar}, our exploration begins with a detailed examination
of the global trivialization of tangent bundles, a defining characteristic
of parallelizable manifolds. This global trivialization allows us to define
fundamental vector fields, which are parallel vector fields with respect to
the trivial connection. Integral curves of these vector lead a local loop
structure, a non-associative analog of a group structure. This structure
enriches the manifold's algebraic framework and opens new avenues for
studying its geometric properties. Furthermore, we introduce a
generalization of the Lie algebra structure on parallelizable manifolds.
This generalization provides a new perspective on the algebraic
underpinnings of these manifolds and offers new insights into their
properties. The interplay between the geometric and algebraic aspects of
parallelizable manifolds, enriched by these Lie-like structures, is a key
focus of this paper. It should be noted that in the 1960's, Kikkawa \cite%
{KikkawaLocal1,KikkawaLocal2} has similarly defined local loop structures
based on geodesics of arbitrary affine connections. However, unlike in the
parallelizable case, such an approach does not allow for a richer algebra
structure.

The structure in Section \ref{sectPar} is defined by a parallelizable
manifold together with a fixed trivialization of the tangent bundle. We
denote this by the \emph{parallelized manifold} triple $(\mathbb{L},%
\mathfrak{l,\rho )}$, where $\mathbb{L}$ is a smooth manifold, $\mathfrak{l}$
a vector space which is the model fiber for the tangent bundle, and $\rho $
is a smooth family of linear isomorphisms from $\mathfrak{l}\ $to each
tangent space $T_{p}\mathbb{L}$. We use the trivialization $\rho ~$to define 
\emph{fundamental vector fields} on $\mathbb{L}$, for each $s\in \mathbb{L}$
given by $\rho _{s}\left( \xi \right) \ $for some $\xi \in \mathfrak{l}.$
These are precisely the parallel vector fields for the flat connection
defined by the trivialization. We also use $\rho $ family of local loop
structures $\circ _{s}$ on $\mathfrak{l}$ and corresponding bracket algebra
structures $b^{\left( s\right) }$, together with a family of trilinear forms 
$a^{\left( s\right) }$ that are defined from differentials of $b^{\left(
s\right) }$ Theorem \ref{thmStructeq} gives an analogue of the Maurer-Cartan
structure equation, while Theorem \ref{thmBrackAssoc} relates the brackets
as infinitesimal commutators of the products $\circ _{s}$. As discussed
above for the case of smooth loops, the bracket function is in general
non-constant on $\mathbb{L}$, and its differential is related to a
skew-symmetrization of infinitesimal associators of the local product on $%
\mathfrak{l}.$

In Section \ref{sectAuto} we then define morphisms of parallelized manifolds
and study how the automorphisms of parallelized manifolds interact with the
algebraic operations defined in Section \ref{sectPar}. In Theorem \ref%
{thmAutomorph} we prove that the automorphism group of $(\mathbb{L},%
\mathfrak{l,\rho )}$ is a finite-dimensional Lie subgroup of the
diffeomorphism group of $\mathbb{L}.$

As a particular example, in Section \ref{sectProdSphere} we consider
explicit trivializations of products of spheres, as given in \cite%
{BruniSpheres,Parton1,Parton2}. As shown by Kervaire in \cite%
{KervaireSpheres}, any product of spheres that contains an odd-dimensional
sphere, is parallelizable. However, we consider a simpler examples of $%
S^{m}\times S^{1}$ and $S^{m}\times N$, where $m$ is arbitrary and $N$ is
any parallelizable manifold. The parallelizability of these spaces follows
immediately from the fact that any sphere is stably parallelizable. We find
that in the case of $\mathbb{L=}S^{m}\times S^{1}$, the standard
trivialization yields a family of bracket algebras $b^{\left( x\right) }$,
for $x\in S^{m}$, each of which is isomorphic to the semidirect sum Lie
algebra $\mathbb{R}^{n}\oplus _{S}\mathbb{R}$. Moreover, considering the
triple product $a^{\left( x\right) }$, the algebraic structure $\left( 
\mathfrak{l},a^{\left( x\right) }\right) $ is that of a Lie triple system 
\cite{JacobsonTriple}. We also explicitly compute the local product $\circ
_{s}$ on $\mathfrak{l}$ in certain special cases and find that the
automorphism group $\Psi \left( \mathbb{L}\right) \cong SO\left( n+1\right)
\times U\left( 1\right) $.

In the more general case of $\mathbb{L=}S^{m}\times N$, we compute the
brackets and the Jacobi identity, finding that for instance, in the case of $%
N=S^{n}\times S^{1}$, for some values of $s\in \mathbb{L}$, then bracket $%
b^{\left( s\right) }$ does not satisfy the Jacobi identity, so generally,
parallelizable manifolds do indeed induce an algebraic structure that is
more general than a Lie algebra or a Lie triple system.

Overall, we see that the framework presented in this paper allows to
associate certain algebraic structures to parallelized manifolds. This paves
the way towards using such algebraic invariants to study and classify
parallelizable manifolds. Moreover, this approach is also easily adaptable
to the broader class of stably parallelizable manifolds as well, as the
example of $S^{m}\times S^{1}$ shows.

\subsection*{Acknowledgements}

This work was supported by the National Science Foundation [DMS-1811754].

\section{Parallelized manifolds}

\setcounter{equation}{0}\label{sectPar}

Suppose  $\mathbb{L}$ is a connected smooth real $n$-dimensional manifold,
with a trivializable tangent bundle $T\mathbb{L}\cong \mathbb{L}\times 
\mathbb{R}^{n}.$ Now given an $n$-dimensional real vector space $\mathfrak{l}%
,$ the global trivialization of $T\mathbb{L}$ induces a diffeomorphism $\rho
:\mathbb{L}\times \mathfrak{l}\longrightarrow T\mathbb{L},$ such that%
\begin{equation}
\mathbb{L}\times \mathfrak{l}\ni \left( p,\xi \right) \mapsto \rho
_{p}\left( \xi \right) \in T_{p}\mathbb{L},  \label{diffeopho}
\end{equation}%
with $\rho _{p}:\mathfrak{l}\longrightarrow T_{p}\mathbb{L}$ being a linear
isomorphism for each $p\in \mathbb{L}$.

\begin{definition}
The \emph{fundamental vector field for }$\xi \in \mathfrak{l}$ of $T\mathbb{L%
}$ is defined to be the smooth vector field $\rho \left( \xi \right) \in
\Gamma \left( T\mathbb{L}\right) $.
\end{definition}

In particular, a choice of basis on $\mathfrak{l}$ induces a global frame of 
$T\mathbb{L},$ or equivalently a section of the corresponding frame bundle $F%
\mathbb{L}.$ A change of trivialization is equivalent to a $GL\left( \mathbb{%
R},n\right) $ gauge transformation of $F\mathbb{L},$ while a change of basis
of $\mathfrak{l}$ corresponds to the right action of $GL\left( \mathbb{R}%
,n\right) $ on the principal bundle $F\mathbb{L}.$

More explicitly, suppose $\left\{ X_{i}\right\} _{i=1}^{n}$ is a frame of $T%
\mathbb{L}.$ Then, assuming we have a basis $\left\{ e_{i}\right\}
_{i=1}^{n} $ on $\mathfrak{l},$ define the corresponding trivialization $%
\rho $ so that $\rho \left( e_{i}\right) =X_{i}$ and hence%
\begin{equation*}
\mathbb{L}\times \mathfrak{l}\ni \left( p,\xi \right) \mapsto \xi ^{i}\left.
X_{i}\right\vert _{p}\in T_{p}\mathbb{L}.
\end{equation*}

\begin{definition}
The triple $\left( \mathbb{L},\mathfrak{l},\rho \right) $, as above, will be
called a \emph{parallelized manifold}, i.e. a parallelizable manifold with a
fixed trivialization.
\end{definition}

\begin{remark}
The above definition is almost equivalent to the notion of \emph{absolute
parallelism}. Absolute parallelism is defined as a smooth family of linear
isomorphisms $\phi _{pq}:T_{q}\mathbb{L}\longrightarrow T_{p}\mathbb{L}$ for
each pair of points $p,q\in \mathbb{L}.$ Given an absolute parallelism and
fixing one of the tangent spaces as th vector space $\mathfrak{l}$ gives the
parallelized manifold definition. Conversely, given the linear maps $\rho
_{p}$ from (\ref{diffeopho}), an absolute parallelism is defined as $\phi
_{pq}=\rho _{p}\rho _{q}^{-1}.$
\end{remark}

\begin{example}
Any Lie group can be parallelized via left translation or right translation.
In that case, $\mathfrak{l}$ corresponds to the Lie algebra, and $\rho $ is
either the left or right translation map.
\end{example}

\begin{example}
Any smooth loop (i.e. a smooth quasigroup with identity and smooth products
and quotients) is parallelizable, in the same way a Lie group \cite%
{GrigorianLoops}. In that case, $\mathfrak{l}$ corresponds to the tangent
algebra at identity. A specific example is $S^{7}$, regarded as the loop of
unit octonions.
\end{example}

\begin{example}
Any orientable $3$-dimensional manifold is parallelizable. For example, $%
S^{2}\times S^{1}$ and $3$-dimensional lens spaces $L\left( p;q\right) $, as
well as $S^{3},$ but this is of course also a Lie group. 
\end{example}

\begin{example}
A $4$-manifold is parallelizable if and only if the Stiefel-Whitney classes $%
w_{1},$ $w_{2}$ vanish, the Euler characteristic vanishes, and the first
Pontryagin class vanishes.
\end{example}

\begin{example}
Products of parallelizable manifolds are parallelizable.
\end{example}

\begin{example}
\label{exProdSphere}As shown by Kervaire in \cite{KervaireSpheres}, a
product of spheres, where at least one of the spheres is odd-dimensional, is
parallelizable. Explicit parallelizations of products of spheres have been
demonstrated in \cite{Parton1,Parton2}.
\end{example}

\begin{remark}
In general, parallelizability of a manifold is a property of the smooth
structure, rather than the underlying topological space. Indeed, as shown by
Milnor \cite[Corollary 1]{MilnorICM}, one may have two smooth manifolds that
are homeomorphic, but one is parallelizable, while the other one is not.
However, in certain situations, the parallelizability does reduce to
topological considerations. A manifold is known as \emph{stably
parallelizable} if its stable tangent bundle $TM\oplus \varepsilon $, where $%
\varepsilon $ is a trivial rank 1 bundle, is trivial. Such manifolds are
also known as \emph{framed manifolds} or $\pi $-manifolds. Stable
parallelizability is in some cases equivalent to the vanishing of certain
characteristic classes \cite{HatcherVB}. As shown in \cite{BredonPimanifolds}%
, an $n$-dimensional stably parallelizable manifold is parallelizable if and
only if $S^{n}$ is parallelizable. This of course happens for $n=1,3,7$.
Hence in these dimensions, parallelizability is equivalent to stable
parallelizability.
\end{remark}

To investigate properties of parallelizable manifolds using approaches from
Riemannian geometry, we may use $\rho $ to define a metric and a connection
on $\mathbb{L}$.

\begin{definition}
Given the triple $\left( \mathbb{L},\mathfrak{l},\rho \right) $ and any
inner product $\left\langle {}\right\rangle $ on $\mathfrak{l},$ define the
metric $g$ on $\mathbb{L}$ as the pullback metric with respect to $\rho
^{-1},$ i.e. for $s\in \mathbb{L},$ and $X_{s},Y_{s}\in T_{s}\mathbb{L},$ 
\begin{equation}
g_{s}\left( X_{s},Y_{s}\right) =\left\langle \rho _{s}^{-1}\left(
X_{p}\right) ,\rho _{s}^{-1}\left( Y_{p}\right) \right\rangle .
\label{gsdef}
\end{equation}
\end{definition}

\begin{definition}
Given the triple $\left( \mathbb{L},\mathfrak{l},\rho \right) ,$ define the
flat connection $\nabla =\left( \rho ^{-1}\right) ^{\ast }d,$ so that 
\begin{equation}
\nabla =\rho \circ d\circ \rho ^{-1}.  \label{flatconn}
\end{equation}
\end{definition}

\begin{lemma}
The metric and the flat connection satisfy the following properties.

\begin{enumerate}
\item $\nabla $ is a metric connection with respect to $g.$

\item Every fundamental vector field has constant norm.
\end{enumerate}
\end{lemma}

\begin{proof}
Both of these properties follow immediately from the definitions.

\begin{enumerate}
\item Suppose $X$ and $Y$ are two vector fields, then 
\begin{eqnarray*}
g\left( \nabla X,Y\right) +g\left( X,\nabla Y\right) &=&g\left( \rho \left(
d\left( \rho ^{-1}X\right) \right) ,Y\right) \\
&&+g\left( X,\rho \left( d\left( \rho ^{-1}Y\right) \right) \right) \\
&=&\left\langle d\left( \rho ^{-1}X\right) ,\rho ^{-1}Y\right\rangle
+\left\langle \rho ^{-1}X,d\left( \rho ^{-1}Y\right) \right\rangle \\
&=&d\left( \left\langle \rho ^{-1}X,\rho ^{-1}Y\right\rangle \right) \\
&=&d\left( g\left( X,Y\right) \right) ,
\end{eqnarray*}%
hence $\nabla $ is metric.

\item Suppose $X=\rho \left( \xi \right) $ is a fundamental vector field,
then it is nowhere zero, and 
\begin{equation*}
d\left( \left\Vert X\right\Vert ^{2}\right) =d\left( g\left( \rho \left( \xi
\right) ,\rho \left( \xi \right) \right) \right) =d\left( \left\langle \xi
,\xi \right\rangle \right) =0,
\end{equation*}%
hence $\left\Vert X\right\Vert $ is constant.
\end{enumerate}
\end{proof}

\begin{remark}
Given a vector field $X\in $ $\Gamma \left( T\mathbb{L}\right) $ and $s\in 
\mathbb{L},$ consider the maximal integral curve $\gamma _{X,s}$ on $L$ of
the vector field $X$ through a point $s\in \mathbb{L}.$ If $\mathbb{L}$ is
compact, then every vector field on $\mathbb{L}$ is complete, and hence $%
\gamma _{X,s}\left( t\right) $ is defined for all $t.$ If $\mathbb{L}$ is
non-compact, consider the global frame $\left\{ X_{i}\right\} $ on $\mathbb{L%
}.$ Each $X_{i}$ is nowhere-vanishing, and given an arbitrary complete
Riemannian metric on $\mathbb{L}$ (which always exists if $\mathbb{L}$ is
connected \cite{NomizuComplete}) we can rescale each of these vector fields
to be of unit norm everywhere. Then, as it is well-known, unit norm vector
fields on a complete Riemannian manifold are complete. Hence, by normalizing
the global frame, we may assume that the fundamental vector fields are
complete. Thus, for non-compact manifolds, without loss of generality, we'll
consider only \emph{complete }trivializations, i.e. those for which the
fundamental vector fields are complete. Then, if $X$ is fundamental, then $%
\gamma _{X,s}\left( t\right) $ is defined for all $t.$ We will assume that $%
\mathbb{L}$ is connected or compact.
\end{remark}

Let $\xi \in \mathfrak{l}$ and define the flow diffeomorphisms $\Phi _{\xi
,t}:\mathbb{L}\longrightarrow \mathbb{L}$ of $\xi $ via

\begin{equation}
\Phi _{\xi ,t}\left( s\right) =\gamma _{\rho \left( \xi \right) ,s}\left(
t\right) ,  \label{PhiX}
\end{equation}%
with $\gamma _{\rho \left( \xi \right) ,s}\left( 0\right) =s.$ By definition
of integral curves, $\Phi _{\xi ,t}\left( s\right) $ is the solution of the
following initial value problem:

\begin{equation}
\left\{ 
\begin{array}{c}
\frac{dp\left( t\right) }{dt}=\left. \rho \left( \xi \right) \right\vert
_{p\left( t\right) } \\ 
p\left( 0\right) =s%
\end{array}%
\right. .  \label{floweq4}
\end{equation}

Since $\rho \left( \xi \right) $ is assumed to be complete, the maps $\Phi
_{\xi ,t}$ are defined for all $t$, and in particular, for each $a\in 
\mathbb{R}$, $\Phi _{a\xi ,t}=\Phi _{\xi ,at}$, and this motivates the
following definition.

\begin{definition}
The product operation $\mu $ on $\mathfrak{l}$ and $\mathbb{L}$ is the map $%
\mu :\mathfrak{l\times }\mathbb{L}\longrightarrow \mathbb{L}$ defined by 
\begin{equation}
\mu \left( \xi ,s\right) =\Phi _{\xi ,1}\left( s\right) ,  \label{mu}
\end{equation}%
where $\xi \in \mathfrak{l},$ $s\in L.$ In particular, we will denote this
product by $\xi \cdot s.$ Define the left and right product maps:

\begin{enumerate}
\item For each $\xi \in \mathfrak{l},$ $L_{\xi }:\mathbb{L}\longrightarrow 
\mathbb{L}$

\item For each $s\in L,$ $R_{s}:\mathfrak{l}\longrightarrow \mathbb{L}.$
\end{enumerate}
\end{definition}

\begin{example}
Suppose $\mathbb{L}$ is a Lie group and $\rho $ is the right-invariant
trivialization of the tangent bundle. Then, $\mathfrak{l}$ is just the
tangent space at identity, i.e. the corresponding Lie algebra, and $\xi
\cdot s=\exp \left( \xi \right) s,$ where $\exp $ is the standard Lie
algebra exponential map.
\end{example}

\begin{lemma}
\label{lemProps}Let $\xi ,\eta \in \mathfrak{l}$. The product operation has
the following properties:

\begin{enumerate}
\item For $t_{1},t_{2}\in \mathbb{R},$ $\left( t_{1}\xi \right) \cdot \left(
\left( t_{2}\xi \right) \cdot s\right) =\left( \left( t_{1}+t_{2}\right) \xi
\right) \cdot s.$

\item For any $s\in L,$ $0\cdot s=s.$

\item $L_{\xi }$ is a diffeomorphism with $L_{\xi }^{-1}=L_{-\xi }.$

\item $\left. dR_{s}\right\vert _{0}=\rho _{s}$.

\item For each $s\in L,$ $R_{s}$ is a local diffeomorphism, and hence the
right quotient $R_{s}^{-1}$ is defined in some neighborhood of $s.$

\item Suppose $p\left( t\right) $ is a curve in $\mathbb{L}$ with $p\left(
0\right) =p$. Assuming that $p/s$ is defined, we have%
\begin{equation}
\left. \frac{d}{dt}R_{t\xi \cdot s}^{-1}\left( p\left( t\right) \right)
\right\vert _{t=0}=\left. \frac{d}{dt}p\left( t\right) /s\right\vert
_{t=0}-\left. \frac{d}{dt}\left( p/s\right) \circ _{s}\left( t\xi \right)
.\right\vert _{t=0}  \label{rinvpt}
\end{equation}
\end{enumerate}
\end{lemma}

\begin{proof}
The first three properties immediately follow from the definition and the
properties of the flow diffeomorphism $\Phi _{\xi ,t}$. For the fourth item,
let $\xi \in \mathfrak{l}$ and consider the straight line from the origin $%
t\xi .$ Then, 
\begin{eqnarray*}
\left. dR_{s}\right\vert _{0}\left( \xi \right)  &=&\left. \frac{d}{dt}%
R_{s}\left( t\xi \right) \right\vert _{t=0} \\
&=&\left. \frac{d}{dt}\Phi _{\xi ,t}\left( s\right) \right\vert _{t=0} \\
&=&\rho _{s}\left( \xi \right) ,
\end{eqnarray*}%
by definition of $R_{s}.$

This shows that $\left. dR_{s}\right\vert _{0}=\rho _{s},$ hence is an
isomorphism, and thus by the Inverse Function Theorem is a local
diffeomorphism.

Now for sufficiently small $t$, consider 
\begin{equation*}
p\left( t\right) =\left( R_{t\xi \cdot s}^{-1}\left( p\left( t\right)
\right) \right) \cdot \left( t\xi \cdot s\right)
\end{equation*}%
Then, differentiating both sides, and noting that $\left. R_{t\xi \cdot
s}^{-1}\left( p\left( t\right) \right) \right\vert _{t=0}=p/s$, we have 
\begin{eqnarray*}
\left. \frac{d}{dt}p\left( t\right) \right\vert _{t=0} &=&\left. \frac{d}{dt}%
\left( R_{t\xi \cdot s}^{-1}\left( p\left( t\right) \right) \right) \cdot
\left( t\xi \cdot s\right) \right\vert _{t=0} \\
&=&\left. \frac{d}{dt}\left( R_{t\xi \cdot s}^{-1}\left( p\left( t\right)
\right) \right) \cdot s\right\vert _{t=0}+\left. \frac{d}{dt}\left( \left(
p/s\right) \cdot \left( t\xi \cdot s\right) \right) \right\vert _{t=0}
\end{eqnarray*}%
Applying $\rho _{s}^{-1}$ to both sides, and noting that $\left.
dR_{s}\right\vert _{0}=\rho _{s}$, we find 
\begin{equation*}
\left. \frac{d}{dt}\left( R_{t\xi \cdot s}^{-1}\left( p\left( t\right)
\right) \right) \right\vert _{t=0}=\left. \frac{d}{dt}p\left( t\right)
/s\right\vert _{t=0}-\left. \frac{d}{dt}\left( p/s\right) \circ _{s}\left(
t\xi \right) .\right\vert _{t=0}
\end{equation*}
\end{proof}

\begin{remark}
Lemma \ref{lemProps} shows that for each $s\in \mathbb{L},$ we may define a 
\emph{local loop structure} on\emph{\ }$\mathfrak{l}$. Indeed, suppose $\xi
,\eta \in \mathfrak{l}$ are in a sufficiently small neighborhood of $0\in 
\mathfrak{l}.$ Then, define 
\begin{equation}
\eta \circ _{s}\xi =\left( \eta \cdot \left( \xi \cdot s\right) \right)
/s\in \mathfrak{l.}  \label{stildeprod}
\end{equation}%
With respect to this product, we see that $0\in \mathfrak{l}$ is a two-sided
identity element. The right quotient with respect to $\circ _{s}$ is defined
by 
\begin{equation}
\xi /_{s}\eta =\left( \xi \cdot s\right) /\left( \eta \cdot s\right) ,
\label{srq}
\end{equation}%
while the left quotient is given by 
\begin{equation}
\xi \backslash _{s}\eta =\left( \xi \backslash \left( \eta \cdot s\right)
\right) /s.  \label{slq}
\end{equation}%
Equivalently, we may define a product on a neighborhood of $s$. Given $p,q$
in a sufficiently small neighborhood of $s\in \mathbb{L}$, define 
\begin{equation}
p\circ _{s}q=\left( p/s\right) \cdot q.  \label{pqs}
\end{equation}%
Note that if $p=\xi \cdot s$ and $q=\eta \cdot s$, then $p\circ _{s}q=\left(
\xi \circ _{s}\eta \right) \cdot s.$ Similarly, the quotients are also
defined.
\end{remark}

\begin{lemma}
The product $\circ _{s}$ is power-associative. In particular, for any $\xi
\in \mathfrak{l},$ the powers $\xi ^{n}=n\xi $ are defined unambiguously and
independently of $s$.
\end{lemma}

\begin{proof}
Let $\xi \in \mathfrak{l.}$ Firstly, for real numbers $t_{1}$ and $t_{2}$
consider 
\begin{eqnarray*}
\left( t_{1}\xi \right) \circ _{s}\left( t_{2}\xi \right) &=&\left( \left(
t_{1}\xi \right) \cdot \left( \left( t_{2}\xi \right) \cdot s\right) \right)
/s \\
&=&\left( t_{1}+t_{2}\right) \xi .
\end{eqnarray*}%
Hence $\xi \circ _{s}\xi =2\xi .$ Now, for third powers 
\begin{equation*}
\left( \xi \circ _{s}\xi \right) \circ _{s}\xi =\left( 2\xi \right) \circ
_{s}\xi =3\xi =\xi \circ _{s}\left( \xi \circ _{s}\xi \right) .
\end{equation*}%
By induction we can see that for any integer power $n$, $\xi ^{n}=n\xi .$
\end{proof}

\begin{theorem}
\label{thmProd}Suppose $\mathbb{L}$ is connected and $s\in \mathbb{L},$ then
any element $p\in \mathbb{L}$ can be written as $p=\xi _{1}\cdot \left( \xi
_{2}\cdot ...\left( \xi _{k}\cdot s\right) \right) ,$ for some finite
sequence $\left\{ \xi _{1},...,\xi _{k}\right\} $ in $\mathfrak{l}.$
\end{theorem}

\begin{proof}
Let $W\subset \mathbb{L}$ be the subset of elements of $\mathbb{L}$ that can
be represented as a product of $s$ by some finite sequence of elements of $%
\mathfrak{l}$ on the left. We will show that $W$ is both closed and open,
and hence $\mathbb{L=}W.$

First note that since for each $p\in \mathbb{L}$, $R_{p}$ is a local
diffeomorphism of a neighborhood of $0\in \mathfrak{l}$ to a neighborhood of 
$p\in \mathbb{L},$ we see that there exists an open neighborhood $V_{p}$ of $%
p,$ where each $q\in V_{p}$ is given by $\eta \cdot p$ for some $\eta \in 
\mathbb{L}.$

Now let $w\in W,$ then by the above, there exists an open neighborhood $%
V_{w} $ of $w,$ where each $q\in V_{w}$ is given by $\eta \cdot w.$ In
particular, $V_{w}\subset W,$ and hence $W$ is open.

To show that $W$ is closed, consider a converging sequence $%
p_{i}\longrightarrow p$ with each $p_{i}\in W.$ Then, for $N$ large enough,
we will have $p_{N}\in V_{p}.$ In other words, we will have $p_{N}=\eta
\cdot p,$ for some $\eta \in \mathfrak{l}.$ This implies that $p=\left(
-\eta \right) \cdot p_{N}\in W,$ and thus $W$ is closed.
\end{proof}

\begin{remark}
The above result shows that given a fixed $s\in \mathbb{L},$ any other
element of $\mathbb{L}$ may be reached in a finite number of
\textquotedblleft steps\textquotedblright , i.e. left multiplications by
elements of $\mathfrak{l}.$ Thus, we may define a discrete \textquotedblleft
distance\textquotedblright\ function $d\left( p,q\right) $ for points $%
p,q\in \mathbb{L}$ given by the minimum number of steps needed to reach $q$
from $p.$
\end{remark}

Given a trivialization $\rho $ and fixing a point $s\in \mathbb{L},$ we
induce a bracket on $\mathfrak{l}.$ Let $\xi ,\eta \in \mathfrak{l}.$ Then, 
\begin{equation}
\left[ \xi ,\eta \right] ^{\left( \rho ,s\right) }=-\rho _{s}^{-1}\left(
\left. \left[ \rho \left( \xi \right) ,\rho \left( \eta \right) \right]
\right\vert _{s}\right) .  \label{lbrack}
\end{equation}%
Note that here we have the negative sign to be compatible with \cite%
{GrigorianLoops} and also to follow the same convention as right-invariant
vector fields on Lie groups. As before, if $\left\{ X_{i}\right\} _{i=1}^{n}$
is a frame of $T\mathbb{L},$ suppose the brackets of the frame elements are
given by 
\begin{equation}
\left[ X_{i},X_{j}\right] =c_{\ ij}^{k\ }X_{k},  \label{sectbrack}
\end{equation}%
where $c_{\ ij}^{k}$ are smooth $\mathbb{F}$-valued functions on $M.$ Then, 
\begin{eqnarray}
\left[ \xi ,\eta \right] ^{\left( \rho ,s\right) } &=&-\rho _{s}^{-1}\left(
\left. \left[ \xi ^{i}X_{i},\eta ^{j}X_{j}\right] \right\vert _{s}\right) .
\\
&=&-\rho _{s}^{-1}\left( \xi ^{i}\eta ^{j}c_{\ ij}^{k}\left( s\right) \left.
X_{k}\right\vert _{s}\right) ,
\end{eqnarray}%
and therefore 
\begin{equation}
\left[ \xi ,\eta \right] ^{\left( \rho ,s\right) }=-\xi ^{i}\eta ^{j}c_{\
ij}^{k}\left( s\right) e_{k}  \label{lbrack2}
\end{equation}%
Define the bracket function 
\begin{equation}
b^{\left( \rho \right) }:\mathbb{L}\longrightarrow \Lambda ^{2}\mathfrak{l}%
^{\ast }\otimes \mathfrak{l}  \label{bfunc}
\end{equation}%
given by 
\begin{equation*}
b^{\left( \rho \right) }\left( s\right) \left( \xi ,\eta \right) =\left[ \xi
,\eta \right] ^{\left( \rho ,s\right) }
\end{equation*}%
for any $s\in \mathbb{L}$ and $\xi ,\eta \in \mathfrak{l}.$ We can also
define the $\func{ad}$ map:%
\begin{equation*}
\func{ad}_{\xi }^{\left( \rho ,s\right) }:\mathfrak{l}\longrightarrow 
\mathfrak{l}
\end{equation*}%
by 
\begin{equation*}
\func{ad}_{\xi }^{\left( \rho ,s\right) }\left( \eta \right) =\left[ \xi
,\eta \right] ^{\left( \rho ,s\right) }.
\end{equation*}

\begin{definition}
Let the \emph{Maurer-Cartan form} $\theta ^{\left( \rho \right) }\in \Omega
^{1}\left( \mathbb{L},\mathfrak{l}\right) $ be an $\mathfrak{l}$-valued $1$%
-form, defined by 
\begin{equation}
\left. \theta ^{\left( \rho \right) }\right\vert _{s}\left( X\right) =\rho
_{s}^{-1}\left( X\right)  \label{MCform1}
\end{equation}%
for any $X\in T_{s}M,$ where $s\in M.$ In particular, for any $\eta \in 
\mathfrak{l},$ given a fundamental vector field $\rho \left( \eta \right) ,$ 
\begin{equation}
\theta ^{\left( \rho \right) }\left( \rho \left( \eta \right) \right) =\eta .
\label{MCform1a}
\end{equation}
\end{definition}

In particular, we can now rewrite (\ref{lbrack}) as 
\begin{equation}
b^{\left( \rho \right) }\left( \theta ^{\left( \rho \right) }\left( X\right)
,\theta ^{\left( \rho \right) }\left( Y\right) \right) =-\theta ^{\left(
\rho \right) }\left( \left[ X,Y\right] \right) .  \label{lbrack1}
\end{equation}

\begin{remark}
To simplify notation, we will drop $\left( \rho \right) $ from $b$ and $%
\theta .$ We still need to remember that these objects depend on the
trivialization $\rho ,$ but once the trivialization is fixed, it is not
necessary to keep writing it.
\end{remark}

\begin{theorem}[Structural Equation]
\label{thmStructeq}The $1$-form $\theta $ satisfies the following relation. 
\begin{equation}
d\theta -\frac{1}{2}b\left( \theta ,\theta \right) =0,  \label{structeq}
\end{equation}%
where $b\left( \theta ,\theta \right) $ is the bracket of $\mathfrak{l}$%
-valued $1$-forms such that for any $X,Y\in T_{p}\mathbb{L}$, $\frac{1}{2}%
\left. b\left( \theta ,\theta \right) \right\vert _{s}\left( X,Y\right) =%
\left[ \theta \left( X\right) ,\theta \left( Y\right) \right] ^{\left(
s\right) }.$
\end{theorem}

\begin{proof}
It is sufficient to check this for fundamental sections of $T\mathbb{L}.$
Let $X=\rho \left( \xi \right) $ and $Y=\rho \left( \eta \right) $ for $\xi
,\eta \in \mathfrak{l}.$ Then from (\ref{MCform1}), 
\begin{eqnarray*}
\left( d\theta \right) \left( X,Y\right) &=&X\left( \theta \left( Y\right)
\right) -Y\left( \theta \left( X\right) \right) -\theta \left( \left[ X,Y%
\right] \right) \\
&=&X\left( \eta \right) -Y\left( \xi \right) -\theta \left( \left[ X,Y\right]
\right) \\
&=&b\left( \theta \left( X\right) ,\theta \left( Y\right) \right) ,
\end{eqnarray*}%
where we have used (\ref{lbrack1}) and the fact that $\eta $ and $\xi $ are
constant.
\end{proof}

Applying the exterior derivative to the structural equation, we find the
generalization of the Jacobi identity:

\begin{theorem}[Generalized Jacobi identity]
The bracket $b$ and the Maurer-Cartan form $\theta $ satisfy:%
\begin{equation}
b\left( \theta ,b\left( \theta ,\theta \right) \right) =\left( db\right)
\left( \theta ,\theta \right) \text{,}  \label{jacid}
\end{equation}%
where exterior derivatives are implied.
\end{theorem}

Define the \emph{skew-associator} map $a:\mathbb{L}\longrightarrow \Lambda
^{2}\mathfrak{l}^{\ast }\otimes \mathfrak{l}^{\ast }\otimes \mathfrak{l}$
such that for $s\in \mathbb{L}$ and $\xi ,\eta ,\gamma \in \mathfrak{l,}$ 
\begin{equation}
a^{\left( s\right) }\left( \xi ,\eta ,\gamma \right) =\left. d_{\rho \left(
\gamma \right) }b\left( \xi ,\eta \right) \right\vert _{p}.  \label{ap1}
\end{equation}%
Suppose $\xi ,\eta ,\gamma \in \mathfrak{l}.$ Then, substituting $\rho
\left( \xi \right) ,\rho \left( \eta \right) ,\rho \left( \gamma \right) \in
\Gamma \left( T\mathbb{L}\right) $ into (\ref{jacid}) and evaluating at a
point $s\in \mathbb{L}$ gives 
\begin{eqnarray}
\left[ \xi ,\left[ \eta ,\gamma \right] ^{\left( s\right) }\right] ^{\left(
s\right) }+\left[ \eta ,\left[ \gamma ,\xi \right] ^{\left( s\right) }\right]
^{\left( s\right) }+\left[ \gamma ,\left[ \xi ,\eta \right] ^{\left(
s\right) }\right] ^{\left( s\right) } &=&a^{\left( s\right) }\left( \xi
,\eta ,\gamma \right)  \label{jacid2} \\
&&+a^{\left( s\right) }\left( \eta ,\gamma ,\xi \right)  \notag \\
&&+a^{\left( s\right) }\left( \gamma ,\xi ,\eta \right) .  \notag
\end{eqnarray}%
We thus see that the obstruction to satisfying the Jacobi identity is
precisely the cyclic permutation of $a$ or equivalently $db.$ In particular,
if $b$ is constant, then $\mathfrak{l}$ is a Lie algebra. In this case, the
bracket is independent of $p,$ and $\mathfrak{l}$ has a unique bracket that
satisfies the Jacobi identity.

Both the bracket and the skew-associator maps can be related to the product $%
\circ _{s}$on $\mathfrak{l}.$ In particular, as we see below, the bracket is
the infinitesimal commutator of $\circ _{s}$ and $a^{\left( s\right) }$ is a
skew-symmetrization of the infinitesimal associator of $\circ _{s}$.

\begin{theorem}
\label{thmBrackAssoc}Let $\xi ,\eta ,\gamma \in \mathfrak{l},$ and $s\in 
\mathbb{L}$, then 
\begin{subequations}
\begin{eqnarray}
\left[ \xi ,\eta \right] ^{\left( s\right) } &=&\left. \frac{d^{2}}{%
dt_{1}dt_{2}}\left( \left( t_{1}\xi \right) \circ _{s}\left( t_{2}\eta
\right) -\left( t_{2}\eta \right) \circ _{s}\left( t_{1}\xi \right) \right)
\right\vert _{t_{1}=t_{2}=0}  \label{brackcomm} \\
a^{\left( s\right) }\left( \xi ,\eta ,\gamma \right) &=&\left[ \xi ,\eta
,\gamma \right] ^{\left( s\right) }-\left[ \eta ,\xi ,\gamma \right]
^{\left( s\right) }\text{,}  \label{aassoc}
\end{eqnarray}%
\end{subequations}%
where%
\begin{equation*}
\left[ \xi ,\eta ,\gamma \right] ^{\left( s\right) }=\left. \frac{d^{3}}{%
dt_{1}dt_{2}dt_{3}}\left( \left( t_{1}\xi \right) \circ _{s}\left( \left(
t_{2}\eta \right) \circ _{s}\left( t_{3}\gamma \right) \right) -\left(
\left( t_{1}\xi \right) \circ _{s}\left( t_{2}\eta \right) \right) \circ
_{s}\left( t_{3}\gamma \right) \right) \right\vert _{t_{1}=t_{2}=t_{3}=0}.
\end{equation*}
\end{theorem}

\begin{proof}
By definition (\ref{lbrack}), we have 
\begin{equation}
\left[ \xi ,\eta \right] ^{\left( s\right) }=-\rho _{s}^{-1}\left( \left. %
\left[ X,Y\right] \right\vert _{s}\right) .
\end{equation}%
where $X=\rho \left( \xi \right) $ and $Y=\rho \left( \eta \right) $ are
vector fields on $\mathbb{L}.$ However, the bracket of two vector fields $X$
and $Y$ is given by%
\begin{equation}
\left[ X,Y\right] _{s}=\left. \frac{d}{dt}\left( \left( \Phi _{\xi
,t}^{-1}\right) _{\ast }\left( Y_{\Phi _{\xi ,t}\left( s\right) }\right)
\right) \right\vert _{t=0},  \label{vecbracket}
\end{equation}%
where $\Phi _{\xi ,t}$ is the flow generated by $X,$ and hence, $\Phi _{\xi
,t}\left( s\right) =\left( t\xi \right) \cdot s.$ Therefore 
\begin{equation*}
Y_{\Phi _{\xi ,t}\left( s\right) }=\rho _{\left( t\xi \cdot s\right) }\left(
\eta \right) .
\end{equation*}%
Hence 
\begin{equation}
\left( \Phi _{\xi ,t}^{-1}\right) _{\ast }\left( Y_{\Phi _{\xi ,t}\left(
s\right) }\right) =\left( L_{t\xi }^{-1}\right) _{\ast }\rho _{\left( t\xi
\cdot s\right) }\left( \eta \right) ,  \label{Phinegt}
\end{equation}%
and from Lemma \ref{lemProps}, $L_{t\xi }^{-1}=L_{-t\xi }$. Then,
differentiating, we find that 
\begin{equation*}
\left. \frac{d}{dt}\left( \left( L_{-t\xi }\right) _{\ast }\rho _{\left(
t\xi \cdot s\right) }\left( \eta \right) \right) \right\vert _{t=0}=\left. 
\frac{d}{dt}\left( \rho _{\left( t\xi \cdot s\right) }\left( \eta \right)
\right) \right\vert _{t=0}-\left. \frac{d}{dt}\left( \left( L_{t\xi }\right)
_{\ast }\rho _{\left( s\right) }\left( \eta \right) \right) \right\vert
_{t=0}.
\end{equation*}%
Also from Lemma \ref{lemProps}, we know that $\rho _{\left( s\right)
}=\left. \left( R_{s}\right) _{\ast }\right\vert _{0}.$ Hence, since for
some parameter $t_{2}$, $t_{2}\eta $ is a path in $\mathfrak{l}$ through $0$
with tangent vector $\eta $ at $0$, we can rewrite 
\begin{eqnarray*}
\rho _{\left( t\xi \cdot s\right) }\left( \eta \right) &=&\left. \left(
R_{t\xi \cdot s}\right) _{\ast }\right\vert _{0}\left( \eta \right) \\
&=&\left. \frac{d}{dt_{2}}\left( \left( t_{2}\eta \right) \cdot \left( t\xi
\cdot s\right) \right) \right\vert _{t_{2}=0} \\
&=&\left. \frac{d}{dt_{2}}R_{s}\left( \left( t_{2}\eta \right) \circ
_{s}\left( t\xi \right) \right) \right\vert _{t_{2}=0}
\end{eqnarray*}%
and similarly, 
\begin{eqnarray*}
\left( L_{t\xi }\right) _{\ast }\rho _{\left( s\right) }\left( \eta \right)
&=&\left. \left( L_{t\xi }\right) _{\ast }\right\vert _{s}\left. \left(
R_{s}\right) _{\ast }\right\vert _{0}\left( \eta \right) \\
&=&\left. \frac{d}{dt_{2}}\left( t\xi \right) \cdot \left( \left( t_{2}\eta
\right) \cdot s\right) \right\vert _{t_{2}=0} \\
&=&\left. \frac{d}{dt_{2}}R_{s}\left( \left( t\xi \right) \circ _{s}\left(
t_{2}\eta \right) \right) \right\vert _{t_{2}=0}
\end{eqnarray*}%
Overall, 
\begin{equation*}
\left. \left[ X,Y\right] \right\vert _{s}=\left( R_{s}\right) _{\ast }\left. 
\frac{d^{2}}{dt_{1}dt_{2}}\left( \left( t_{2}\eta \right) \circ _{s}\left(
t_{1}\xi \right) -\left( t_{1}\xi \right) \circ _{s}\left( t_{2}\eta \right)
\right) \right\vert _{t_{1}=t_{2}=0},
\end{equation*}%
and therefore, we indeed get (\ref{brackcomm}).

Now consider 
\begin{equation*}
\left. \frac{d^{3}}{dt_{1}dt_{2}dt_{3}}\left( t_{1}\xi \right) \circ
_{\left( t_{3}\gamma \cdot s\right) }\left( t_{2}\eta \right) \right\vert
_{t_{1}=t_{2}=t_{3}=0}.
\end{equation*}%
Expanding, we get 
\begin{eqnarray*}
\left( t_{1}\xi \right) \circ _{\left( t_{3}\gamma \cdot s\right) }\left(
t_{2}\eta \right) &=&\left( R_{t_{3}\gamma \cdot s}\right) ^{-1}\left(
\left( t_{1}\xi \right) \cdot \left( \left( t_{2}\eta \right) \cdot \left(
t_{3}\gamma \cdot s\right) \right) \right) \\
&=&\left( R_{t_{3}\gamma \cdot s}\right) ^{-1}\left( \left( \left( t_{1}\xi
\right) \circ _{s}\left( \left( t_{2}\eta \right) \circ _{s}\left(
t_{3}\gamma \right) \right) \right) \cdot s\right) .
\end{eqnarray*}%
Using (\ref{rinvpt}) with $p\left( t_{3}\right) =\left( \left( t_{1}\xi
\right) \circ _{s}\left( \left( t_{2}\eta \right) \circ _{s}\left(
t_{3}\gamma \right) \right) \right) \cdot s$, we find 
\begin{eqnarray*}
\left. \frac{d}{dt_{3}}\left( R_{t_{3}\gamma \cdot s}\right) ^{-1}\left(
\left( \left( t_{1}\xi \right) \circ _{s}\left( \left( t_{2}\eta \right)
\circ _{s}\left( t_{3}\gamma \right) \right) \right) \cdot s\right)
\right\vert _{t_{3}=0} &=&\left. \frac{d}{dt_{3}}\left( t_{1}\xi \right)
\circ _{s}\left( \left( t_{2}\eta \right) \circ _{s}\left( t_{3}\gamma
\right) \right) \right\vert _{t_{3}=0} \\
&&-\left. \frac{d}{dt_{3}}\left( \left( t_{1}\xi \right) \circ _{s}\left(
t_{2}\eta \right) \right) \circ _{s}\left( t_{3}\gamma \right) \right\vert
_{t_{3}=0}.
\end{eqnarray*}%
Hence, 
\begin{equation*}
\left. \frac{d^{3}}{dt_{1}dt_{2}dt_{3}}\left( t_{1}\xi \right) \circ
_{\left( t_{3}\gamma \cdot s\right) }\left( t_{2}\eta \right) \right\vert
_{t_{1}=t_{2}=t_{3}=0}=\left[ \xi ,\eta ,\gamma \right] ^{\left( s\right) }.
\end{equation*}%
Consider now 
\begin{equation*}
\left. d_{\rho \left( \gamma \right) }b\left( \xi ,\eta \right) \right\vert
_{s}=a^{\left( s\right) }\left( \xi ,\eta ,\gamma \right) ,
\end{equation*}%
where we use the definition of $a^{\left( s\right) }$ (\ref{ap1}). However,
from (\ref{brackcomm}), 
\begin{eqnarray*}
\left. d_{\rho \left( \gamma \right) }b\left( \xi ,\eta \right) \right\vert
_{s} &=&\left. \frac{d}{dt_{3}}\left[ \xi ,\eta \right] ^{t_{3}\gamma \cdot
s}\right\vert _{t_{3}=0} \\
&=&\left. \frac{d^{3}}{dt_{1}dt_{2}dt_{3}}\left( \left( t_{1}\xi \right)
\circ _{\left( t_{3}\gamma \cdot s\right) }\left( t_{2}\eta \right) -\left(
t_{2}\eta \right) \circ _{\left( t_{3}\gamma \cdot s\right) }\left( t_{2}\xi
\right) \right) \right\vert _{t_{1}=t_{2}=t_{3}=0} \\
&=&\left[ \xi ,\eta ,\gamma \right] ^{\left( s\right) }-\left[ \eta ,\xi
,\gamma \right] ^{\left( s\right) }.
\end{eqnarray*}%
So indeed, 
\begin{equation*}
a^{\left( s\right) }\left( \xi ,\eta ,\gamma \right) =\left[ \xi ,\eta
,\gamma \right] ^{\left( s\right) }-\left[ \eta ,\xi ,\gamma \right]
^{\left( s\right) }.
\end{equation*}
\end{proof}

Consider now the connection $\nabla $ from (\ref{flatconn}). Its torsion $%
T\left( X,Y\right) $ for vector fields $X$ and $Y$ is given by 
\begin{equation*}
T\left( X,Y\right) =\nabla _{X}Y-\nabla _{Y}X-\left[ X,Y\right] ,
\end{equation*}%
so in particular for $X=\rho \left( \xi \right) $ and $Y=\rho \left( \eta
\right) ,$ for $\xi ,\eta \in \mathfrak{l}$, we have 
\begin{equation*}
T\left( \rho \left( \xi \right) ,\rho \left( \eta \right) \right) =-\left[
\rho \left( \xi \right) ,\rho \left( \eta \right) \right] .
\end{equation*}%
Hence, using (\ref{lbrack}), we find that at a point $s\in \mathbb{L}$, we
have 
\begin{equation}
T\left( \rho \left( \xi \right) ,\rho \left( \eta \right) \right) _{s}=\rho
_{s}\left( \left[ \xi ,\eta \right] ^{\left( \rho ,s\right) }\right) .
\label{Texp}
\end{equation}%
Now let $Z=\rho \left( \gamma \right) $ for $\gamma \in \mathfrak{l,}$ hence 
\begin{equation}
g_{s}\left( T\left( \rho \left( \xi \right) ,\rho \left( \eta \right)
\right) _{s},\rho _{s}\left( \gamma \right) \right) =\left\langle \left[ \xi
,\eta \right] ^{\left( \rho ,s\right) },\gamma \right\rangle .  \label{gT}
\end{equation}%
Recall that a torsion of a metric connection is said to be totally
skew-symmetric if for any vector fields $X,Y,Z,$ 
\begin{equation}
g\left( T\left( X,Y\right) ,Z\right) =-g\left( Y,T\left( X,Z\right) \right) .
\label{Tskew}
\end{equation}%
From (\ref{gT}), we see that the torsion of the flat connection $\nabla $ is
totally skew-symmetric if and only if for any $s\in \mathbb{L}$, 
\begin{equation}
\left\langle \left[ \xi ,\eta \right] ^{\left( \rho ,s\right) },\gamma
\right\rangle =-\left\langle \eta ,\left[ \xi ,\gamma \right] ^{\left( \rho
,s\right) }\right\rangle .  \label{admet}
\end{equation}%
Equivalently, for any $\xi \in \mathfrak{l}$ and any $s\in \mathbb{L},$ the
map $\func{ad}_{\xi }^{\left( \rho ,s\right) }$ is skew-adjoint with respect
to the inner product on $\mathfrak{l}.$

Moreover, consider $\nabla T.$ Then, 
\begin{eqnarray*}
\left( \nabla _{\rho \left( \gamma \right) }T\right) \left( \rho \left( \xi
\right) ,\rho \left( \eta \right) \right) &=&\nabla _{\rho \left( \gamma
\right) }\left( T\left( \rho \left( \xi \right) ,\rho \left( \eta \right)
\right) \right) \\
&=&\nabla _{\rho \left( \gamma \right) }\left( \rho \left( b\left( \xi ,\eta
\right) \right) \right) \\
&=&\rho \left( d_{\rho \left( \gamma \right) }\left( b\left( \xi ,\eta
\right) \right) \right) \\
&=&\rho \left( a\left( \xi ,\eta ,\gamma \right) \right) .
\end{eqnarray*}%
A classical theorem by Cartan and Schouten \cite{cartan1926riem} (and \cite%
{AgricolaFriedrichFlat} for a more modern treatment) states the following.

\begin{theorem}[\protect\cite{cartan1926riem}]
\label{thmCS23}Let $\left( M,g\right) $ be a simply-connected, complete
irreducible Riemannian manifold with a flat metric connection $\nabla $ with
a non-zero totally skew-symmetric torsion $T.$ Then

\begin{enumerate}
\item If $\nabla T=0,$ then $M$ is isometric to a compact simple Lie groups.

\item If $\nabla T\neq 0$, then $M$ is isometric to $S^{7}.$
\end{enumerate}
\end{theorem}

In our formulation, we have the following corollary.

\begin{corollary}
Suppose $\left( \mathbb{L},\mathfrak{l},\rho \right) $ is a parallelized
manifold. If for $s\in \mathbb{L}$, there exists a smooth family of inner
products $\left\langle \cdot ,\cdot \right\rangle _{s}$ on $\mathfrak{l}$
such that the bracket $\left[ \cdot ,\cdot \right] ^{\left( \rho ,s\right) }$
is skew-adjoint with respect to it, then there exists a parallelization $%
\tilde{\rho}$ of $\mathbb{L}$ and an inner product $\left\langle \cdot
,\cdot \right\rangle $ on $\mathfrak{l},$ such that the Riemannian manifold $%
\left( \mathbb{L},g\right) $ with $g=\left\langle \tilde{\rho}^{-1},\tilde{%
\rho}^{-1}\right\rangle $ is isometric to a product of compact simple Lie
groups and copies of $S^{7}.$
\end{corollary}

\begin{proof}
Suppose for each $s\in \mathbb{L},$ $\left\langle \cdot ,\cdot \right\rangle
_{s}$ is such that for any $\xi ,\eta ,\gamma \in \mathfrak{l},$ 
\begin{equation*}
\left\langle \left[ \xi ,\eta \right] ^{\left( \rho ,s\right) },\gamma
\right\rangle _{s}=-\left\langle \eta ,\left[ \xi ,\gamma \right] ^{\left(
\rho ,s\right) }\right\rangle _{s}.
\end{equation*}%
Then, choosing any inner product $\left\langle \cdot ,\cdot \right\rangle $
on $\mathfrak{l},$ we can write $\left\langle \xi ,\eta \right\rangle
_{s}=\left\langle Q_{s}^{-1}\xi ,Q_{s}^{-1}\eta \right\rangle ,$ for some
map $Q_{s}\in GL\left( \mathfrak{l}\right) .$ Since the family of inner
products is smooth, we obtain a smooth family of maps $Q_{s}.$ Now define a
new parallelization of $\mathbb{L}$ given by $\tilde{\rho}_{s}=\rho
_{s}\circ Q_{s}.$ Then, 
\begin{eqnarray*}
\left. g\left( X,Y\right) \right\vert _{s} &=&\left\langle \tilde{\rho}%
_{s}^{-1}X,\tilde{\rho}_{s}^{-1}Y\right\rangle \\
&=&\left\langle Q_{s}^{-1}\rho _{s}^{-1}X,Q_{s}^{-1}\rho
_{s}^{-1}Y\right\rangle \\
&=&\left\langle \rho _{s}^{-1}X,\rho _{s}^{-1}Y\right\rangle _{s}.
\end{eqnarray*}%
In particular, note that by definition (\ref{lbrack}) 
\begin{eqnarray*}
\left[ \xi ,\eta \right] ^{\left( \tilde{\rho},s\right) } &=&-\tilde{\rho}%
_{s}^{-1}\left( \left. \left[ \tilde{\rho}\left( \xi \right) ,\tilde{\rho}%
\left( \eta \right) \right] \right\vert _{s}\right) \\
&=&-Q_{s}^{-1}\rho _{s}^{-1}\left( \left. \left[ \rho \left( Q\left( \xi
\right) \right) ,\rho \left( Q\left( \eta \right) \right) \right]
\right\vert _{s}\right) \\
&=&Q_{s}^{-1}\left[ Q_{s}\left( \xi \right) ,Q_{s}\left( \eta \right) \right]
^{\left( \rho ,s\right) }
\end{eqnarray*}%
and hence 
\begin{eqnarray*}
\left\langle \left[ \xi ,\eta \right] ^{\left( \tilde{\rho},s\right)
},\gamma \right\rangle &=&\left\langle Q_{s}^{-1}\left[ Q_{s}\left( \xi
\right) ,Q_{s}\left( \eta \right) \right] ^{\left( \rho ,s\right) },\gamma
\right\rangle \\
&=&\left\langle \left[ Q_{s}\left( \xi \right) ,Q_{s}\left( \eta \right) %
\right] ^{\left( \rho ,s\right) },Q_{s}\left( \gamma \right) \right\rangle \\
&=&-\left\langle Q_{s}\left( \eta \right) ,\left[ Q_{s}\left( \xi \right)
,Q_{s}\left( \gamma \right) \right] ^{\left( \rho ,s\right) }\right\rangle \\
&=&-\left\langle \eta ,\left[ \xi ,\gamma \right] ^{\left( \tilde{\rho}%
,s\right) }\right\rangle .
\end{eqnarray*}%
In particular, we now see that with $g$ and $\nabla $ defined using $\tilde{%
\rho},$ the torsion of $\nabla $ satisfies (\ref{Tskew}), and hence by
Theorem \ref{thmCS23}, $\left( \mathbb{L},g\right) $ is isometric to a
product of compact simple Lie groups and copies of $S^{7}.$
\end{proof}

\section{Automorphisms}

\setcounter{equation}{0} \label{sectAuto}

\begin{definition}
The category of (complete, connected) parallelized manifolds is defined in
the following way. An object in this category is the triple $\left( \mathbb{L%
},\mathfrak{l},\rho \right) ,$ where $\mathbb{L}$ is an $n$-dimensional
parallelizable connected manifold, $\mathfrak{l}$ is an $n$-dimensional
vector space, and $\rho $ is the trivialization map, such that the
fundamental vector fields are complete.

Suppose $\left( \mathbb{L}_{1},\mathfrak{l}_{1},\rho _{1}\right) $ and $%
\left( \mathbb{L}_{2},\mathfrak{l}_{2},\rho _{2}\right) $ are two
parallelized manifolds, then a morphism between them is a smooth map $h:%
\mathbb{L}_{1}\longrightarrow \mathbb{L}_{2}\ $for which there exists a
linear map $h^{\prime }:\mathfrak{l}_{1}\longrightarrow \mathfrak{l}_{2}$
such that the following diagram commutes

\begin{equation}
\begin{tikzcd} T_{p}\mathbb{L}_{1} \arrow[r, "\left. h_{\ast }\right\vert
_{p}"] & T_{h(p)} \mathbb{L}_{2} \\ \mathfrak{l}_{1} \arrow[u, " \left( \rho
_{1}\right) _{p}"'] \arrow[r, "h'"'] & \mathfrak{l}_{2} \arrow[u, " \left(
\rho _{2}\right) _{h\left( p\right)}"'] \end{tikzcd}  \label{figmorphism}
\end{equation}%
We will refer to $\left( h,h^{\prime }\right) $ as a compatible pair.
\end{definition}

\begin{remark}
the notation $h^{\prime }$ is appropriate since this map is closely related
to the derivative of $h.$ Any smooth $h$ will induce an invertible linear
map $h_{p}^{\prime }=\rho _{f\left( p\right) }^{-1}\circ \left. h_{\ast
}\right\vert _{p}\circ \rho _{p}:\mathfrak{l}_{1}\longrightarrow \mathfrak{l}%
_{2},$ however in general it will depend on $p.$ The key point of the above
definition is that for a morphism of parallelized manifolds, $h_{p}^{\prime }
$ is independent of $p$. A morphism between parallelized manifolds is a
special case of a vector bundle morphism. Indeed, consider the tangent
bundles $T\mathbb{L}_{1}$ and $T\mathbb{L}_{2}.$ Then, the bundle map $%
\left( h,\rho _{2}\circ h^{\prime }\circ \rho _{1}^{-1}\right) :T\mathbb{L}%
_{1}\longrightarrow T\mathbb{L}_{2}$ covers $h:\mathbb{L}_{1}\longrightarrow 
\mathbb{L}_{2}$ and induces a linear map of each fiber.
\end{remark}

\begin{example}
Suppose we have a parallelized manifold $\left( \mathbb{L},\mathfrak{l},\rho
_{1}\right) $, then any linear isomorphism of the vector space $\mathfrak{l}$
induces an isomorphism of parallelized manifolds. Indeed, suppose $h^{\prime
}:\mathfrak{l}\longrightarrow \mathfrak{l}$ is an invertible linear map.
Then, the compatible pair $\left( \func{id},h^{\prime }\right) $, where $%
\func{id}$ is the identity map on $\mathbb{L}$, give an isomorphism from $%
\left( \mathbb{L},\mathfrak{l},\rho _{1}\right) $ to $\left( \mathbb{L},%
\mathfrak{l},\rho _{2}\right) ,$ where $\rho _{2}=\rho _{1}\circ \left(
h^{\prime }\right) ^{-1}.$ This just corresponds to a $GL\left( n\right) $
transformation of the global frame on $T\mathbb{L}.$
\end{example}

\begin{lemma}
\label{lemPseudoProd}Suppose $\left( \mathbb{L}_{1},\mathfrak{l}_{1},\rho
_{1}\right) $ and $\left( \mathbb{L}_{2},\mathfrak{l}_{2},\rho _{2}\right) $
are two parallelized manifolds. A pair $\left( h,h^{\prime }\right) $ of a
smooth map $h:\mathbb{L}_{1}\longrightarrow \mathbb{L}_{2}$ and a linear map 
$h^{\prime }:\mathfrak{l}_{1}\longrightarrow \mathfrak{l}_{2}$ satisfy the
definition of a morphism of parallelized manifolds if and only if for any $%
s\in \mathbb{L}_{1}$ and $\xi \in \mathfrak{l}_{1}$, 
\begin{equation}
h\left( \xi \cdot s\right) =h^{\prime }\left( \xi \right) \cdot h\left(
s\right) .  \label{fhpseudo}
\end{equation}
\end{lemma}

\begin{proof}
Let $\xi \in \mathfrak{l}_{1},$ and let $s\in \mathbb{L}.$ Recall that $\xi
\cdot s=\gamma _{\rho _{1}\left( \xi \right) ,s}\left( 1\right) $ is the
solution $x\left( 1\right) $ at $t=1$ of the initial value problem 
\begin{equation*}
\left\{ 
\begin{array}{c}
\frac{dx\left( t\right) }{dt}=\rho _{1}\left( \xi \right) _{x\left( t\right)
} \\ 
x\left( 0\right) =s%
\end{array}%
\right. .
\end{equation*}%
Then, applying $h_{\ast }$ to the equation, we get 
\begin{eqnarray}
\left. h_{\ast }\right\vert _{x\left( t\right) }\left( \frac{dx\left(
t\right) }{dt}\right) &=&\frac{d\left( h\left( x\left( t\right) \right)
\right) }{dt}  \label{fstardxtild} \\
&=&\left. h_{\ast }\right\vert _{x\left( t\right) }\left( \rho _{1}\left(
\xi \right) _{x\left( t\right) }\right)  \notag
\end{eqnarray}%
\qquad

Suppose now the pair $\left( h,h^{\prime }\right) $ satisfies the definition
of a morphism of parallelized manifolds. Using the property of $h^{\prime }$
from (\ref{figmorphism}), we have 
\begin{equation*}
\left. h_{\ast }\right\vert _{x\left( t\right) }\left( \rho _{1}\left( \xi
\right) _{x\left( t\right) }\right) =\rho _{2}\left( h^{\prime }\left( \xi
\right) \right) _{h\left( x\left( t\right) \right) }.
\end{equation*}%
Therefore, we see that $y=h\left( x\left( t\right) \right) $ satisfies the
equation%
\begin{equation}
\left\{ 
\begin{array}{c}
\frac{dy\left( t\right) }{dt}=\rho _{2}\left( h^{\prime }\left( \xi \right)
\right) _{y\left( t\right) } \\ 
y\left( 0\right) =h\left( s\right) 
\end{array}%
\right.   \label{eqhx}
\end{equation}%
and overall we see that $y\left( t\right) =\gamma _{\rho _{2}\left( \xi
\right) ,h\left( s\right) }\left( t\right) $ For $t=1,$ we hence have%
\begin{equation*}
h\left( \xi \cdot s\right) =h^{\prime }\left( \xi \right) \cdot h\left(
s\right) .
\end{equation*}

Conversely, suppose (\ref{fhpseudo}) holds for any $s\in \mathbb{L}_{1}$ and 
$\xi \in \mathfrak{l}_{1}.$ Then, we have 
\begin{equation}
h\left( t\xi \cdot s\right) =\left( th^{\prime }\left( \xi \right) \right)
\cdot h\left( s\right) .
\end{equation}%
Differentiating both sides, we find 
\begin{equation*}
\left. h_{\ast }\right\vert _{x\left( t\right) }\left( \rho _{1}\left( \xi
\right) _{x\left( t\right) }\right) =\rho _{2}\left( h^{\prime }\left( \xi
\right) \right) _{y\left( t\right) },
\end{equation*}%
where $x\left( t\right) =t\xi \cdot s$ and $y\left( t\right) =\left(
th^{\prime }\left( \xi \right) \right) \cdot h\left( s\right) =h\left(
x\left( t\right) \right) $. Setting $t=0$, we find that the pair $\left(
h,h^{\prime }\right) $ indeed satisfies the definition of a morphism of
parallelizable manifolds.
\end{proof}

\begin{example}
In the case of Lie groups, the definition of a morphism of parallelized
manifolds is equivalent to taking the pair of a Lie group homomorphism and
its differential at identity.
\end{example}

\begin{definition}
\label{defAutPar}Suppose $\left( \mathbb{L},\mathfrak{l},\rho \right) $ is a
parallelized manifold. Then, an automorphism of $\left( \mathbb{L},\mathfrak{%
l},\rho \right) $ is a diffeomorphism $h:\mathbb{L}\longrightarrow \mathbb{L}
$ for which there exists an invertible linear map $h^{\prime }:\mathfrak{l}%
\longrightarrow \mathfrak{l}$, such that the following diagram commutes%
\begin{equation}
\begin{tikzcd} T_{p}\mathbb{L} \arrow[r, "\left. h_{\ast }\right\vert _{p}"]
& T_{h(p)} \mathbb{L} \\ \mathfrak{l} \arrow[u, " \left( \rho \right)
_{p}"'] \arrow[r, "h'"'] & \mathfrak{l} \arrow[u, " \left( \rho \right)
_{h\left( p\right)}"'] \end{tikzcd}  \label{figmorphism2}
\end{equation}
\end{definition}

\begin{remark}
Given a trivialization $\rho ,$ we may define the trivial connection $\nabla 
$ on $T\mathbb{L}$ via%
\begin{equation}
\nabla =\rho \circ d\circ \rho ^{-1},  \label{nablatriv}
\end{equation}%
where $\rho $ is interpreted now as a map $\rho :C^{\infty }\left( \mathbb{L}%
,\mathfrak{l}\right) \longrightarrow \Gamma \left( T\mathbb{L}\right) .$
Given a diffeomorphism $h$ of $\mathbb{L},$ the pullback connection $h^{\ast
}\nabla $ is defined via 
\begin{eqnarray}
h^{\ast }\nabla &=&h_{\ast }^{-1}\circ \nabla \circ h_{\ast }
\label{pullbackconn} \\
&=&h_{\ast }^{-1}\circ \rho \circ d\circ \rho ^{-1}\circ h_{\ast }.  \notag
\end{eqnarray}%
Setting $h^{\prime }=\rho ^{-1}\circ h_{\ast }\circ \rho ,$ we get 
\begin{equation*}
h^{\ast }\nabla =\rho ^{-1}\circ \left( h^{\prime }\right) ^{-1}\circ d\circ
h\circ \rho .
\end{equation*}%
Hence we see that $h^{\ast }\nabla =\nabla $ if and only if $h^{\prime }$ is
constant and equivalently $h$ is an automorphism of $\left( \mathbb{L},%
\mathfrak{l},\rho \right) .$ In other words, $h$ is a connection-preserving
diffeomorphism, or as it is also known as, an \emph{affine diffeomorphism}
for the trivial connection $\nabla .$
\end{remark}

The set of automorphisms of $\left( \mathbb{L},\mathfrak{l},\rho \right) $
is clearly non-empty, since the identity map is always an automorphism, and
it's clearly a Lie group, with the Lie group structure induced from $\func{%
Diff}\left( \mathbb{L}\right) $. Let us denote it by $\Psi \left( \mathbb{L}%
\mathfrak{,l,\rho }\right) ,$ or just $\Psi $ if there is no ambiguity. Each 
$h\in \Psi \left( \mathbb{L}\mathfrak{,l,\rho }\right) $ defines a
corresponding linear isomorphism $h^{\prime }$ of $\mathfrak{l},$ so denote
this subgroup of $GL\left( \mathfrak{l}\right) $ by $\Psi ^{\prime }\left( 
\mathbb{L},\mathfrak{l,\rho }\right) ,$ or just $\Psi ^{\prime }$ if there
is no ambiguity. Under an automorphism, a fundamental vector field $\rho
\left( \xi \right) $ is pushed forward to $\rho \left( h^{\prime }\left( \xi
\right) \right) .$ The map $h\mapsto h^{\prime }$ is a Lie group
homomorphism, so $\Psi ^{\prime }$ is a Lie group of dimension at most $n^{2}
$ where $n=\dim \mathbb{L}.$

Let $\pi :\Psi \longrightarrow \Psi ^{\prime }$ be the homomorphism that
takes $h$ to $h^{\prime }.$ Let $\mathfrak{p}$ and $\mathfrak{p}^{\prime }$
be the Lie algebras of $\Psi $ and $\Psi ^{\prime },$ respectively, then at
least for small $t$, $\pi \left( \exp _{\mathfrak{p}}\left( t\eta \right)
\right) =\exp _{\mathfrak{p}^{\prime }}\left( t\pi _{\ast }\eta \right) $
for $\eta \in \mathfrak{p}.$

\begin{theorem}
\label{thmAutomorph}Suppose $\mathbb{L}$ is a connected parallelizable
manifold of dimension $n$. The automorphism group $\Psi =\Psi \left( \mathbb{%
L}\mathfrak{,l,\rho }\right) $ of $\left( \mathbb{L},\mathfrak{l},\rho
\right) $ is a finite-dimensional Lie group of dimension at most $n^{2}+n.$
\end{theorem}

\begin{proof}
The Lie group structure of $\Psi $ is clearly induced from $\func{Diff}%
\left( \mathbb{L}\right) ,$ so it is sufficient to show that $\Psi $ is a
finite-dimensional manifold. Fix any $s\in \mathbb{L}$, and define the
smooth map $\Psi \longrightarrow \Psi ^{\prime }\times \mathbb{L}$ given for
any $h\in \Psi $ by 
\begin{equation}
h\mapsto \left( h^{\prime },h\left( s\right) \right)  \label{hprimes}
\end{equation}%
where $h^{\prime }\in GL\left( \mathfrak{l}\right) $ is the corresponding
element of $\Psi ^{\prime }.$ We will show that this map is injective.
Indeed, suppose $\tilde{h}\in \Psi $ also maps to $\left( h^{\prime
},h\left( s\right) \right) .$ Then, consider $k=h^{-1}\tilde{h}\mapsto
\left( \func{id},s\right) .$ Recall from Theorem \ref{thmProd} that any
element $p\in \mathbb{L}$ can be written as a finite product $p=\xi
_{1}\cdot \left( \xi _{2}\cdot ...\left( \xi _{m}\cdot s\right) \right) $
for $\xi _{i}\in \mathfrak{l}.$ Then, applying $k$ to $p$ using Lemma \ref%
{lemPseudoProd} gives 
\begin{equation*}
k\left( p\right) =k^{\prime }\left( \xi _{1}\right) \cdot \left( k^{\prime
}\left( \xi _{2}\right) \cdot ...\left( k^{\prime }\left( \xi _{m}\right)
\cdot k\left( s\right) \right) \right) =p,
\end{equation*}%
since $k^{\prime }=\func{id}$ and $k\left( s\right) =s$. Hence, $k=\func{id}$%
, and hence (\ref{hprimes}) is an injective map.

Suppose $h_{t}\in \Psi \subset \func{Diff}\left( \mathbb{L}\right) ,$ is a
smooth $1$-parameter family with $h_{0}=\func{id}.$ Then, for a fixed $s\in 
\mathbb{L},$ $h_{t}\left( s\right) $ is a smooth curve on $\mathbb{L},$ with 
$u_{s}=\left. \frac{d}{dt}h_{t}\left( s\right) \right\vert _{t=0}\in T_{s}%
\mathbb{L}.$ The vector field $u\in \Gamma \left( T\mathbb{L}\right) $ is
then the tangent vector to the curve $h_{t}\in \Psi $ at $t=0$.

Recall from Theorem \ref{thmProd} that any element $p\in \mathbb{L}$ can be
written as a finite product $p=\xi _{1}\cdot \left( \xi _{2}\cdot ...\left(
\xi _{m}\cdot s\right) \right) $ for $\xi _{i}\in \mathfrak{l},$ then 
\begin{eqnarray}
u_{p} &=&\left. \frac{d}{dt}h_{t}\left( p\right) \right\vert _{t=0}  \notag
\\
&=&\left. \frac{d}{dt}h_{t}\left( \xi _{1}\cdot \left( \xi _{2}\cdot
...\left( \xi _{m}\cdot s\right) \right) \right) \right\vert _{t=0}  \notag
\\
&=&\left. \frac{d}{dt}h_{t}^{\prime }\left( \xi _{1}\right) \cdot \left(
h_{t}^{\prime }\left( \xi _{2}\right) \cdot ...\left( h_{t}^{\prime }\left(
\xi _{m}\right) \cdot h_{t}\left( s\right) \right) \right) \right\vert _{t=0}
\notag \\
&=&\left. \frac{d}{dt}h_{t}^{\prime }\left( \xi _{1}\right) \cdot \left( \xi
_{2}\cdot ...\left( \xi _{m}\cdot s\right) \right) \right\vert _{t=0}+...
\label{up} \\
&&+\left. \frac{d}{dt}\xi _{1}\cdot \left( \xi _{2}\cdot ...\left(
h_{t}^{\prime }\left( \xi _{m}\right) \cdot s\right) \right) \right\vert
_{t=0}  \notag \\
&&+\left. \frac{d}{dt}\xi _{1}\cdot \left( \xi _{2}\cdot ...\left( \xi
_{m}\cdot \left( h_{t}\left( s\right) /s\cdot s\right) \right) \right)
\right\vert _{t=0}  \notag
\end{eqnarray}%
In particular, consider the linear map 
\begin{equation}
u\mapsto \left( \left. \frac{d}{dt}h_{t}^{\prime }\right\vert _{t=0},\left. 
\frac{d}{dt}h_{t}\left( s\right) /s\right\vert _{t=0}\right) \in \mathfrak{p}%
^{\prime }\times \mathfrak{l.}  \label{umap}
\end{equation}%
This is the differential of the map (\ref{hprimes}) at identity. From (\ref%
{up}) we see that if $\left. \frac{d}{dt}h_{t}^{\prime }\right\vert _{t=0}=0$
and $\left. \frac{d}{dt}h_{t}\left( s\right) /s\right\vert _{t=0}=0$, then $%
u=0$. Thus, we see that the map (\ref{umap}) is injective. Similarly, by
translation, the differential of (\ref{hprimes}) will be injective at any $%
h\in \Psi .$ Therefore, (\ref{hprimes}) is an injective smooth immersion,
and thus a smooth embedding. Hence, $\Psi \ $is smoothly embedded in $\Psi
^{\prime }\times \mathbb{L},$ which is at most $\left( n^{2}+n\right) $%
-dimensional.
\end{proof}

Suppose $s\in \mathbb{L}$ is fixed, then recall from (\ref{stildeprod}) that
the product $\circ _{s}$on $\mathfrak{l}$ is defined by $\eta \circ _{s}\xi
=\left( \eta \cdot \left( \xi \cdot s\right) \right) /s$ for sufficiently
small $\eta ,\xi \in \mathfrak{l}.$ Now for $h\in \Psi $ and the
corresponding $h^{\prime }\in \Psi ^{\prime }$, consider 
\begin{eqnarray}
h^{\prime }\left( \eta \circ _{s}\xi \right)  &=&h\left( \eta \cdot \left(
\xi \cdot s\right) \right) /h\left( s\right)   \notag \\
&=&h^{\prime }\left( \eta \right) \cdot \left( h^{\prime }\left( \xi \right)
\cdot h\left( s\right) \right) /h\left( s\right)   \notag \\
&=&h^{\prime }\left( \eta \right) \circ _{h\left( s\right) }h^{\prime
}\left( \xi \right) .  \label{hsprod}
\end{eqnarray}%
Equivalently, we can write 
\begin{equation}
h^{\prime }\left( \eta \circ _{s}\xi \right) \cdot h\left( s\right)
=h^{\prime }\left( \eta \right) \cdot \left( h^{\prime }\left( \xi \right)
\cdot h\left( s\right) \right) ,  \label{hcomp}
\end{equation}%
or 
\begin{equation}
h^{\prime }\left( \eta \circ _{s}\xi \right) \circ _{s}\left( h\left(
s\right) /s\right) =h^{\prime }\left( \eta \right) \circ _{s}\left(
h^{\prime }\left( \xi \right) \circ _{s}\left( h\left( s\right) /s\right)
\right) ,  \label{hcomp2}
\end{equation}

\begin{remark}
In the theory of loops and quasigroups, a map $h^{\prime }$ that satisfies
the property $h^{\prime }\left( xy\right) A=h^{\prime }\left( x\right)
\left( h^{\prime }\left( y\right) A\right) $ is known as a \emph{%
pseudoautomorphism} with \emph{companion} $A$. Using this language, we can
say that $h^{\prime }$ has \emph{companion }$h\left( s\right) /s$ with
respect to $s.$ The group $\func{Aut}\left( \mathfrak{l},\circ _{s}\right) $
of automorphisms of $\circ _{s}$ is then defined as the subset of elements
of $\Psi ^{\prime }$ which have companion $0.$ This then corresponds to the
stabilizer $\func{Stab}_{\Psi }\left( s\right) $ of $s.$
\end{remark}

The group $\Psi $ also acts on the bracket $\left[ \cdot ,\cdot \right]
^{\left( s\right) }.$ As we noted above, on Lie groups, the map $h^{\prime }$
is the differential of $h$ at identity, and we know this is a Lie algebra
isomorphism. However, in the more general case, the map $h^{\prime }$ is an
algebra isomorphism from $\left( \mathfrak{l},\left[ \cdot ,\cdot \right]
^{\left( s\right) }\right) $ to $\left( \mathfrak{l},\left[ \cdot ,\cdot %
\right] ^{\left( h\left( s\right) \right) }\right) .$ 

\begin{theorem}
\label{thmhbrack}Suppose $s\in \mathbb{L}$ and $h\in \Psi $, then for any $%
\xi ,\eta \in \mathfrak{l},$ 
\begin{equation}
h^{\prime }\left( \left[ \xi ,\eta \right] ^{\left( s\right) }\right) =\left[
h^{\prime }\left( \xi \right) ,h^{\prime }\left( \eta \right) \right]
^{\left( h\left( s\right) \right) }.  \label{hbrack}
\end{equation}
\end{theorem}

\begin{proof}
Indeed, from Theorem \ref{thmBrackAssoc} and using (\ref{hsprod}), we obtain 
\begin{eqnarray*}
h^{\prime }\left( \left[ \xi ,\eta \right] ^{\left( s\right) }\right)
&=&\left. \frac{d^{2}}{dt_{1}dt_{2}}\left( h^{\prime }\left( \left( t_{1}\xi
\right) \circ _{s}\left( t_{2}\eta \right) \right) -h^{\prime }\left( \left(
t_{2}\eta \right) \circ _{s}\left( t_{1}\xi \right) \right) \right)
\right\vert _{t_{1}=t_{2}=0} \\
&=&\left. \frac{d^{2}}{dt_{1}dt_{2}}\left( \left( t_{1}h^{\prime }\left( \xi
\right) \right) \circ _{h\left( s\right) }\left( t_{2}h^{\prime }\left( \eta
\right) \right) -\left( \left( t_{2}h^{\prime }\left( \eta \right) \right)
\circ _{h\left( s\right) }\left( t_{1}h^{\prime }\left( \xi \right) \right)
\right) \right) \right\vert _{t_{1}=t_{2}=0} \\
&=&\left[ h^{\prime }\left( \xi \right) ,h^{\prime }\left( \eta \right) %
\right] ^{h\left( s\right) }.
\end{eqnarray*}
\end{proof}

\section{Products of spheres}

\setcounter{equation}{0} \label{sectProdSphere}As mentioned in\ Example \ref%
{exProdSphere}, a large family of parallelizable manifolds that are not Lie
groups or Moufang loops is given by products of spheres where one factor is
an odd-dimensional sphere. As it is well-known, the only parallelizable
spheres are $S^{1},S^{3},S^{7}$, which correspond to unit norm sets of
complex numbers, quaternions, and octonions. In particular, $S^{1}$ and $%
S^{3}$ are both Lie groups, while $S^{7}$ is the Moufang loop of unit
octonions. In \cite{BruniSpheres,Parton1,Parton2} explicit parallelizations
on products of spheres have been computed and this will serve as an
important example of the framework developed in this paper.

In particular, let us consider in detail the case of $S^{m}\times S^{1}$
from \cite{Parton1}. Consider $S^{m}$ as embedded in $\mathbb{R}^{m+1}$ in a
standard way%
\begin{equation}
S^{m}=\left\{ \left( x_{1},...,x_{m+1}\right) \in \mathbb{R}%
^{m+1}:\left\vert x\right\vert ^{2}=x_{1}^{2}+...+x_{m+1}^{2}=1\right\} .
\label{Smdef}
\end{equation}%
Consider the standard coordinate frame $\left\{ \partial _{x_{i}}\right\} $
on $\mathbb{R}^{m+1}$. Then, the orthogonal projections $\left\{
M_{i}\right\} $ of these vector fields onto $S^{m}$ are given by 
\begin{equation}
M_{i}=\partial _{x_{i}}-x_{i}M,  \label{Mi}
\end{equation}%
where $M=\sum_{i=1}^{m+1}x_{i}\partial _{x_{i}}$ is the normal vector field
to $S^{m}\subset \mathbb{R}^{m+1}.$ Also consider $\phi $ as a coordinate on 
$S^{1}$ with the associated vector field $\partial _{\phi }.$ Then, a
parallelization of $S^{m}\times S^{1}$ is given by the following.

\begin{theorem}
Consider $S^{m}\subset \mathbb{R}^{m+1}$ and for $i=1,...,n+1$, define%
\begin{equation}
f_{i}=M_{i}+x_{i}\partial _{\phi }.  \label{bi}
\end{equation}%
Then, the set $\left\{ f_{i}\right\} $ is global frame on $S^{m}\times S^{1}.
$ In particular, it is orthonormal with respect to the product metric and
satisfies 
\begin{equation}
\left[ f_{i},f_{j}\right] =x_{i}f_{j}-x_{j}f_{i}.  \label{bbrack}
\end{equation}
\end{theorem}

Now let $\mathbb{L}=S^{m}\times S^{1}$, so that each $s\in \mathbb{L}$ is
given by $s=\left( x,\phi \right) $ for $x\in S^{m}\subset \mathbb{R}^{m+1}$
and $\phi \in S^{1}.$ Also suppose $\mathfrak{l}$ is an $\left( m+1\right) $%
-dimensional real vector space with an inner product $\left\langle \cdot
,\cdot \right\rangle $ and an orthonormal basis $\left\{ e_{i}\right\} .$
Then, for each $s=\left( x,\phi \right) \in \mathbb{L},$ define 
\begin{equation}
\begin{array}{c}
\rho _{s}:\mathfrak{l}\longrightarrow T_{s}\mathbb{L} \\ 
e_{i}\mapsto \left. f_{i}\right\vert _{s}%
\end{array}%
.  \label{rhosprod}
\end{equation}%
In particular, if $\xi =\xi ^{i}e_{i}\in \mathfrak{l},$ then $\rho \left(
\xi \right) =\xi ^{i}f_{i}.$

Consider the derivatives of the coordinate functions in the directions of
the basis vectors $\left\{ f_{i}\right\} .$ We have 
\begin{eqnarray*}
f_{i}\left( x_{j}\right)  &=&\left( \partial _{x_{i}}-x_{i}M\right) x_{j} \\
&=&\delta _{ij}-x_{i}x_{j} \\
f_{i}\left( \phi \right)  &=&x_{i}
\end{eqnarray*}%
In particular, define $c_{ij}=f_{i}\left( x_{j}\right) =\delta
_{ij}-x_{i}x_{j}.$

The bracket on $\mathfrak{l}$ is defined from (\ref{bbrack}) via (\ref%
{lbrack}). In particular, for the basis $\left\{ e_{i}\right\} $ on $%
\mathfrak{l}$ and $s=\left( x,\phi \right) \in \mathbb{L}$, we have 
\begin{equation}
b_{ij}^{\left( s\right) }=:\left[ e_{i},e_{j}\right] ^{\left( s\right)
}=e_{i}x_{j}-e_{j}x_{i}.  \label{spherebrack}
\end{equation}%
Equivalently, given any $\xi ,\eta \in \mathfrak{l}$, we have 
\begin{equation}
\left[ \xi ,\eta \right] ^{\left( s\right) }=\xi \left\langle \eta
,x\right\rangle -\eta \left\langle \xi ,x\right\rangle .
\label{spherebrack2}
\end{equation}%
In particular, note that if both $\xi $ and $\eta $ are orthogonal to $x$,
then $\left[ \xi ,\eta \right] ^{\left( s\right) }=0$ and for any $\xi \in 
\mathfrak{l},$ $\left[ \xi ,x\right] ^{\left( s\right) }=\xi $.

\begin{lemma}
\label{lemSnS1Jac}For each fixed $s$, $\mathfrak{l}$ with the bracket $\left[
\cdot ,\cdot \right] ^{\left( s\right) }$ defines a Lie algebra that is
isomorphic to a semidirect sum $\mathbb{R}^{n}\oplus _{S}\mathbb{R},$ with
bracket $\left[ \left( \xi ,\lambda \right) ,\left( \eta ,\mu \right) \right]
=\left[ \mu \xi -\lambda \eta ,0\right] .$
\end{lemma}

\begin{proof}
First let us show that $\left[ \cdot ,\cdot \right] ^{\left( s\right) }$, as
given by (\ref{spherebrack2}) defines a Lie algebra on $\mathfrak{l}.$ We
just need to check the Jacobi identity. Indeed, for $\gamma ,\xi ,\eta \in 
\mathfrak{l},$ 
\begin{eqnarray*}
\left[ \gamma ,\left[ \xi ,\eta \right] ^{\left( s\right) }\right] ^{\left(
s\right) } &=&\left[ \gamma ,\xi \right] ^{\left( s\right) }\left\langle
\eta ,x\right\rangle -\left[ \gamma ,\eta \right] ^{\left( s\right)
}\left\langle \xi ,x\right\rangle  \\
&=&\gamma \left\langle \xi ,x\right\rangle \left\langle \eta ,x\right\rangle
-\xi \left\langle \gamma ,x\right\rangle \left\langle \eta ,x\right\rangle 
\\
&&-\gamma \left\langle \eta ,x\right\rangle \left\langle \xi ,x\right\rangle
+\eta \left\langle \gamma ,x\right\rangle \left\langle \xi ,x\right\rangle 
\\
&=&\eta \left\langle \gamma ,x\right\rangle \left\langle \xi ,x\right\rangle
-\xi \left\langle \gamma ,x\right\rangle \left\langle \eta ,x\right\rangle 
\\
\left[ \xi ,\left[ \eta ,\gamma \right] ^{\left( s\right) }\right] ^{\left(
s\right) } &=&\gamma \left\langle \xi ,x\right\rangle \left\langle \eta
,x\right\rangle -\eta \left\langle \xi ,x\right\rangle \left\langle \gamma
,x\right\rangle  \\
\left[ \eta ,\left[ \gamma ,\xi \right] ^{\left( s\right) }\right] ^{\left(
s\right) } &=&\xi \left\langle \eta ,x\right\rangle \left\langle \gamma
,x\right\rangle -\gamma \left\langle \eta ,x\right\rangle \left\langle \xi
,x\right\rangle 
\end{eqnarray*}%
Adding these terms together shows that indeed the Jacobi identity is
satisfied. Now consider the linear map $F:\mathfrak{l}\longrightarrow 
\mathbb{R}^{n}\oplus \mathbb{R}$, given by 
\begin{equation*}
F\left( \xi \right) =\left( \xi -\left\langle \xi ,x\right\rangle
x,\left\langle \xi ,x\right\rangle \right) .
\end{equation*}%
This is clearly a linear isomorphism. Then, 
\begin{eqnarray*}
F\left( \left[ \xi ,\eta \right] ^{\left( s\right) }\right)  &=&F\left( \xi
\right) \left\langle \eta ,x\right\rangle -F\left( \eta \right) \left\langle
\xi ,x\right\rangle  \\
&=&\left( \xi -\left\langle \xi ,x\right\rangle x,\left\langle \xi
,x\right\rangle \right) \left\langle \eta ,x\right\rangle  \\
&&-\left( \eta -\left\langle \eta ,x\right\rangle x,\left\langle \eta
,x\right\rangle \right) \left\langle \xi ,x\right\rangle  \\
&=&\left( \xi \left\langle \eta ,x\right\rangle -\eta \left\langle \xi
,x\right\rangle ,0\right)  \\
&=&\left[ \left( \xi -\left\langle \xi ,x\right\rangle x,\left\langle \xi
,x\right\rangle \right) ,\left( \eta -\left\langle \eta ,x\right\rangle
x,\left\langle \eta ,x\right\rangle \right) \right]  \\
&=&\left[ F\left( \xi \right) ,F\left( \eta \right) \right] .
\end{eqnarray*}%
Hence, this is a Lie algebra homomorphism, and thus the Lie algebra $\left( 
\mathfrak{l},\left[ \cdot ,\cdot \right] ^{\left( s\right) }\right) $ is
isomorphic to $\mathbb{R}^{n}\oplus _{S}\mathbb{R}.$
\end{proof}

Taking the derivative of the bracket in the direction of $f_{i}$ gives us%
\begin{eqnarray*}
d_{f_{i}}\left( \left[ \xi ,\eta \right] ^{\left( s\right) }\right)  &=&\xi
\eta ^{j}c_{ij}-\eta \xi ^{j}c_{ij} \\
&=&\xi \eta _{i}-\eta \xi _{i}-x_{i}\left[ \xi ,\eta \right] ^{\left(
s\right) }
\end{eqnarray*}%
In particular, this shows that 
\begin{equation}
d_{f_{i}}b_{jk}^{\left( s\right) }=:a_{ijk}^{\left( s\right) }=e_{k}\delta
_{ij}-e_{j}\delta _{ik}-x_{i}b_{jk}^{\left( s\right) }.  \label{sphereaijk}
\end{equation}%
Hence, for $\xi ,\eta ,\gamma \in \mathfrak{l}$, 
\begin{equation}
a^{\left( s\right) }\left( \xi ,\eta ,\gamma \right) =\gamma \left(
\left\langle \xi ,\eta \right\rangle +\left\langle x,\xi \right\rangle
\left\langle x,\eta \right\rangle \right) -\eta \left( \left\langle \xi
,\gamma \right\rangle +\left\langle x,\xi \right\rangle \left\langle
x,\gamma \right\rangle \right) .  \label{spherecircas}
\end{equation}

Recall the notion of a \emph{Lie triple system \cite{JacobsonTriple}: }this
is a vector space equipped with a skew-symmetric triple product $\left[
\cdot ,\cdot ,\cdot \right] $, that is skew-symmetric in two entries, has a
vanishing sum of cyclic permutations(i.e. satisfies the Jacobi identity),
and also that the linear map defined by fixing two of the entries of $\left[
\cdot ,\cdot ,\cdot \right] $ is a derivation of the triple product.
Consider now $\mathfrak{l}$ equipped with the triple product $a^{\left(
s\right) }.$ We know that $a^{\left( s\right) }$ is skew-symmetric in the
last two entries. From (\ref{jacid2}), we also see that $a^{\left( s\right) }
$ satisfies the Jacobi identity if and only if $\left( \mathfrak{l}%
,b^{\left( s\right) }\right) $ is a Lie algebra. Therefore, in this case, it
just suffices to check the derivation property to figure out if $\left( 
\mathfrak{l},a^{\left( s\right) }\right) $ is a Lie triple system.

\begin{theorem}
\label{thmTriple}For each $s\in \mathbb{L}$, the vector space $\mathfrak{l}$
equipped with the triple product $a^{\left( s\right) }$ forms a Lie triple
system.
\end{theorem}

\begin{proof}
Since $b^{\left( s\right) }$ satisfies the Jacobi identity, from (\ref%
{jacid2}) we know that 
\begin{equation*}
a^{\left( s\right) }\left( \xi ,\eta ,\gamma \right) +a^{\left( s\right)
}\left( \eta ,\gamma ,\xi \right) +a^{\left( s\right) }\left( \gamma ,\xi
,\eta \right) =0.
\end{equation*}%
To show that $\left( \mathfrak{l},a^{\left( s\right) }\right) $ forms a Lie
triple system, we just need to show that for any $\xi ,\eta \in \mathfrak{l}$%
, the map $a_{\xi ,\eta }^{\left( s\right) }:\mathfrak{l}\longrightarrow 
\mathfrak{l}$ given by $a_{\xi ,\eta }^{\left( s\right) }\left( \gamma
\right) =a^{\left( s\right) }\left( \gamma ,\xi ,\eta \right) $ is a
derivation for $a^{\left( s\right) }$. That is, for any $\xi ,\eta ,\gamma
,u,v\in \mathfrak{l},$ we need to show that 
\begin{eqnarray}
a_{\xi ,\eta }^{\left( s\right) }\left( a^{\left( s\right) }\left(
u,v,\gamma \right) \right) &=&a^{\left( s\right) }\left( a^{\left( s\right)
}\left( u,\xi ,\eta \right) ,v,\gamma \right)  \label{jactrip1} \\
&&+a^{\left( s\right) }\left( u,a^{\left( s\right) }\left( v,\xi ,\eta
\right) ,\gamma \right)  \notag \\
&&+a^{\left( s\right) }\left( u,v,a^{\left( s\right) }\left( \gamma ,\xi
,\eta \right) \right) .  \notag
\end{eqnarray}%
This is a straightforward, but tedious calculation, which we show in
Appendix \ref{secApp}.
\end{proof}

To define products on $\mathfrak{l}$ and $\mathbb{L},$ we need to consider
integral curves of fundamental vector fields. In (\ref{floweq4}) suppose the
curve $p\left( t\right) \in \mathbb{L}$ has coordinates $\left( x\left(
t\right) ,\phi \left( t\right) \right) .$ Suppose $\xi \neq 0$, and let us
rewrite $\rho \left( \xi \right) $ in terms of the coordinate basis:%
\begin{eqnarray*}
\rho \left( \xi \right)  &=&\xi ^{i}f_{i} \\
&=&\xi ^{i}\partial _{x_{i}}-\xi ^{i}x_{i}\sum_{i=1}^{m+1}x_{j}\partial
_{x_{j}}+\xi ^{i}x_{i}\partial _{\phi }.
\end{eqnarray*}%
Setting $\hat{\xi}=\xi ^{i}x_{i}$, we can thus rewrite the flow equation (%
\ref{floweq4}) as 
\begin{equation}
\left\{ 
\begin{array}{c}
\frac{dx\left( t\right) }{dt}=\xi -\hat{\xi}x \\ 
\frac{d\phi \left( t\right) }{dt}=\hat{\xi}%
\end{array}%
\right. ,  \label{floweq5}
\end{equation}%
with initial conditions $\left( x_{0},\phi _{0}\right) $. Without loss of
generality, assume $\left\vert \xi \right\vert =1.$

\begin{lemma}
\label{lemxphisol}The solution of (\ref{floweq5}) exists for all $t\in 
\mathbb{R}$ and, if $\xi =\pm x_{0}$, the solutions are given by 
\begin{subequations}%
\label{x0sol} 
\begin{eqnarray}
x\left( t\right) &=&x_{0}  \label{xtsol0} \\
\phi \left( t\right) &=&\phi _{0}\pm t.  \label{phisol0}
\end{eqnarray}%
\end{subequations}%
If $\xi \neq x_{0}$, then the solution is given by 
\begin{subequations}%
\label{x1sol}%
\begin{eqnarray}
x\left( t\right) &=&\xi \tanh \left( \sigma +t\right) +\tilde{x}_{0}\func{%
sech}\left( \sigma +t\right)  \label{xtsol} \\
\phi \left( t\right) &=&\phi _{0}+\ln \left( \frac{\cosh \left( t+\sigma
\right) }{\cosh \sigma }\right) ,  \label{phisol}
\end{eqnarray}%
\end{subequations}%
where $\tilde{x}_{0}=x_{0}\cosh \sigma -\xi \sinh \sigma $ and $\sigma $ is
such that $\hat{\xi}\left( 0\right) =\tanh \sigma .$
\end{lemma}

\begin{proof}
From (\ref{floweq5}), we see $x\left( t\right) $ is constant if and only if $%
\xi =\pm x_{0}.$In that case, $\hat{\xi}\left( t\right) =\pm 1$, and thus $%
\phi \left( t\right) =\phi _{0}\pm t.$

More generally, taking the inner product of the first equation in (\ref%
{floweq5}) with $\xi $, we obtain an equation for $\hat{\xi}\left( t\right)
: $%
\begin{equation}
\frac{d\hat{\xi}}{dt}=1-\hat{\xi}^{2}.  \label{xihateq}
\end{equation}%
Since $\hat{\xi}^{2}=1$ if and only if $\xi =\pm x_{0},$ suppose $\hat{\xi}%
\left( t\right) =\tanh \sigma \left( t\right) ,$ for some function $\sigma
\left( t\right) .$ Then, (\ref{xihateq}) immediately gives $\sigma \left(
t\right) =t+\sigma $, where $\sigma =\sigma \left( 0\right) ,$ that is, $%
\hat{\xi}\left( 0\right) =\tanh \sigma $.

Let us integrate $\hat{\xi}$, because this will be needed as an integrating
factor for the first equation in (\ref{floweq5}) and also to solve for $\phi
\left( t\right) $ in the second equation in (\ref{floweq5}):%
\begin{equation}
\int^{t}\hat{\xi}\left( \tau \right) d\tau =\ln \left( \cosh \left( \sigma
+t\right) \right) .  \label{intxihat}
\end{equation}%
Hence, we can write the integrating factor as%
\begin{equation*}
M\left( t\right) =\cosh \left( \sigma +t\right)
\end{equation*}%
and the general solution is 
\begin{equation}
x\left( t\right) =\xi \tanh \left( \sigma +t\right) +\tilde{x}_{0}\func{sech}%
\left( \sigma +t\right) ,  \label{xthyp}
\end{equation}%
where%
\begin{equation*}
\tilde{x}_{0}=x_{0}\cosh \sigma -\xi \sinh \sigma .
\end{equation*}%
We also obtain $\phi \left( t\right) $ (\ref{phisol}) by directly
integrating the second equation of (\ref{floweq5}):%
\begin{equation*}
\phi \left( t\right) =\phi _{0}+\ln \left( \frac{\cosh \left( t+\sigma
\right) }{\cosh \sigma }\right)
\end{equation*}%
As expected on a compact manifold, the integral curve exists for all values
of $t$.
\end{proof}

Overall, we see the geometric picture of how the integral curve evolves. For
the non-trivial case, suppose $\xi \neq \pm x_{0}$, then from (\ref{xthyp}),
we see that $x\left( t\right) $ always lies in the plane spanned by $x_{0}$
and $\xi $, so the curve is along the great circle of $S^{m}$ given by the
plane spanned by $x_{0}$ and $\xi $. If $\alpha \left( t\right) \ $is the
angle between $x\left( t\right) $ and $\xi ,$ then $\hat{\xi}\left( t\right)
=\cos \alpha \left( t\right) $. In particular, then 
\begin{eqnarray*}
\cos \alpha \left( t\right) &=&\tanh \left( \sigma +t\right) \\
\sin \alpha \left( t\right) &=&\func{sech}\left( \sigma +t\right) .
\end{eqnarray*}%
In particular, we see that $\cos \alpha \left( t\right) $ is a monotonic
function$,$ with $\cos \alpha \left( t\right) \longrightarrow +1$ as $%
t\longrightarrow +\infty $ and $\cos \alpha \left( t\right) \longrightarrow
-1$ as $t\longrightarrow -\infty $. Moreover, $\sin \alpha \left( t\right) $
is always positive. Therefore, $\alpha \left( t\right) \in \left( 0,\pi
\right) $, with $\alpha \left( t\right) \longrightarrow 0$ as $%
t\longrightarrow +\infty $ and $\alpha \left( t\right) \longrightarrow \pi $
as $t\longrightarrow -\infty $. The equation (\ref{xtsol}) can be rewritten
as 
\begin{equation}
x\left( t\right) =\xi \cos \alpha \left( t\right) +\tilde{x}_{0}\sin \alpha
\left( t\right) .  \label{xtcossin}
\end{equation}%
This shows that in the adapted $\xi $-$\tilde{x}_{0}$ plane, $x\left(
t\right) $ stays in the upper half-plane, moving clockwise as $t$ increases.
In particular, as $t\longrightarrow \pm \infty $, $x\left( t\right)
\longrightarrow \pm \xi $.

\begin{example}
Since $a^{\left( s\right) }$ in this case is non-trivial (\ref{spherecircas}%
), from Theorem \ref{thmBrackAssoc} we see that $\circ _{s}$ is necessarily
non-associative. To illustrate this product $\circ _{s}$ on $\mathfrak{l}$ (%
\ref{stildeprod}) more explicitly, let $s=\left( x,\phi \right) $ and
consider elements unit length $\xi ,\eta \in \mathfrak{l}$, such that they
are orthogonal to each other and to $x$. The more general case is
computationally a bit more involved. Consider first $\left( x_{1}\left(
t\right) ,\phi _{1}\left( t\right) \right) =\left( t\xi \right) \cdot s.$
From (\ref{x1sol}), noting that $\xi \cdot x=0$ implies $\sigma =0,$ we have 
\begin{eqnarray}
x_{1}\left( t\right) &=&\xi \tanh t+x\func{sech}t \\
\phi _{1}\left( t\right) &=&\phi _{0}+\ln \left( \cosh t\right) .
\end{eqnarray}%
Next, consider $\left( x_{2}\left( \tau \right) ,\phi _{2}\left( \tau
\right) \right) =\left( \tau \eta \right) \cdot \left( \left( t\xi \right)
\cdot s\right) .$ By assumption, the initial value $x_{2}\left( 0\right)
=\xi \tanh t+x\func{sech}t$ is orthogonal to $\eta $, so $\sigma $ is again $%
0$. Thus, 
\begin{eqnarray*}
x_{2}\left( \tau \right) &=&\eta \tanh \tau +\xi \tanh t\func{sech}\tau +x%
\func{sech}t\func{sech}\tau \\
\phi _{2}\left( \tau \right) &=&\phi _{0}+\ln \left( \cosh t\cosh \tau
\right) .
\end{eqnarray*}%
To find $\left( \tau \eta \right) \circ _{s}\left( t\xi \right) =\left( \tau
\eta \right) \cdot \left( \left( t\xi \right) \cdot s\right) /s$, we need to
find a positive $t_{2}\in \mathbb{R}$ and $\gamma \in \mathfrak{l}$ of unit
length, such that 
\begin{eqnarray*}
x_{2}\left( \tau \right) &=&\gamma \tanh \left( \sigma _{2}+t_{2}\right)
+\left( x\cosh \sigma _{2}-\gamma \sinh \sigma _{2}\right) \func{sech}\left(
\sigma _{2}+t_{2}\right) \\
\phi _{2}\left( \tau \right) &=&\phi _{0}+\ln \left( \frac{\cosh \left(
t_{2}+\sigma _{2}\right) }{\cosh \sigma _{2}}\right) +2\pi k\text{,}
\end{eqnarray*}%
where $\tanh \sigma _{2}=\gamma \cdot x$ and $k$ is an integer. The freedom
to add an integer multiple of $2\pi k$ stems from the fact that $\phi $ is a
coordinate on the circle. We know that $\gamma $ is a linear combination of $%
\xi ,\eta ,x,$ so let 
\begin{equation}
\gamma =a\xi +b\eta +\left( \tanh \sigma _{2}\right) x.  \label{gamabc}
\end{equation}%
Then, taking inner products of $x_{2}$ with $\xi ,\eta ,x$, and also
comparing $\phi _{2}$, we find the following equations 
\begin{subequations}
\begin{eqnarray}
\tanh t\func{sech}\tau &=&a\left( \tanh \left( t_{2}+\sigma _{2}\right)
-\sinh \sigma _{2}\func{sech}\left( t_{2}+\sigma _{2}\right) \right) \\
\tanh \tau &=&b\left( \tanh \left( t_{2}+\sigma _{2}\right) -\sinh \sigma
_{2}\func{sech}\left( t_{2}+\sigma _{2}\right) \right) \\
\func{sech}t\func{sech}\tau &=&\tanh \sigma _{2}\tanh \left( t_{2}+\sigma
_{2}\right)  \label{exeq3} \\
&&+\func{sech}\sigma _{2}\func{sech}\left( t_{2}+\sigma _{2}\right)  \notag
\\
e^{2\pi k}\cosh t\cosh \tau &=&\frac{\cosh \left( t_{2}+\sigma _{2}\right) }{%
\cosh \sigma _{2}}  \label{exeq4} \\
1 &=&a^{2}+b^{2}+\tanh ^{2}\sigma _{2}.
\end{eqnarray}%
\end{subequations}%
The last equation comes from the fact that $\gamma $ is unit length. However
if we require that $t_{2}\longrightarrow 0$ as $t,\tau \longrightarrow 0$,
which makes sense in order for the product $\circ _{s}$to be defined in the
neighborhood of $0$, (\ref{exeq4}) shows that $k=0$. From (\ref{exeq3}) and (%
\ref{exeq4}) with $k=0$, we find%
\begin{equation*}
\tanh \sigma _{2}\tanh \left( t_{2}+\sigma _{2}\right) +\left( \func{sech}%
\sigma _{2}-\cosh \sigma _{2}\right) \func{sech}\left( t_{2}+\sigma
_{2}\right) =0
\end{equation*}%
and thus 
\begin{equation*}
\tanh \sigma _{2}\left( \sinh \left( t_{2}+\sigma _{2}\right) -\sinh \sigma
_{2}\right) =0.
\end{equation*}%
Since $t_{2}>0$, we find that $\tanh \sigma _{2}=0$, and hence $\sigma
_{2}=0.$ This immediately gives us 
\begin{eqnarray*}
\cosh t_{2} &=&\cosh t\cosh \tau \\
a &=&\frac{\tanh t\func{sech}\tau }{\tanh t_{2}} \\
b &=&\frac{\tanh \tau }{\tanh t_{2}}.
\end{eqnarray*}%
Thus, we conclude that for an orthogonal triple $\eta ,\xi ,x$, we get 
\begin{equation*}
\left( \tau \eta \right) \circ _{s}\left( t\xi \right) =t_{2}\left( a\xi
+b\eta \right) ,
\end{equation*}%
with $t_{2}$, $a,$ and $b$ as given above.
\end{example}

Let us now consider the isomorphisms of $\mathbb{L}$. Suppose $h\in \Psi
\left( \mathbb{L}\right) $ and $h^{\prime }\in \Psi ^{\prime }\left( \mathbb{%
L}\right) $ is the map of $\mathfrak{l}.$ Recall from Lemma \ref%
{lemPseudoProd} that for $\xi \in \mathfrak{l}$ and $s\in \mathbb{L}$, $%
h\left( \xi \cdot s\right) =h^{\prime }\left( \xi \right) \cdot h\left(
s\right) .$

Given $s=\left( x,\phi \right) \in \mathbb{L}$, let the corresponding
components of $h\left( s\right) $ be $h\left( s\right) _{x}$ and $h\left(
s\right) _{\phi }$. From Theorem \ref{thmhbrack}, we then know that for any $%
\xi ,\eta \in \mathfrak{l}$ and $s=\left( x,\phi \right) \in \mathbb{L}$ 
\begin{equation*}
h^{\prime }\left( \left[ \xi ,\eta \right] ^{\left( s\right) }\right) =\left[
h^{\prime }\left( \xi \right) ,h^{\prime }\left( \eta \right) \right]
^{h\left( s\right) }.
\end{equation*}%
However, from (\ref{spherebrack2}), 
\begin{equation}
h^{\prime }\left( \left[ \xi ,\eta \right] ^{\left( s\right) }\right)
=h^{\prime }\left( \xi \right) \left\langle \eta ,x\right\rangle -h^{\prime
}\left( \eta \right) \left\langle \xi ,x\right\rangle
\end{equation}%
and 
\begin{equation}
\left[ h^{\prime }\left( \xi \right) ,h^{\prime }\left( \eta \right) \right]
^{h\left( s\right) }=h^{\prime }\left( \xi \right) \left\langle h^{\prime
}\left( \eta \right) ,h\left( s\right) _{x}\right\rangle -h^{\prime }\left(
\eta \right) \left\langle h^{\prime }\left( \xi \right) ,h\left( s\right)
_{x}\right\rangle .
\end{equation}%
Hence, in particular, for any $\eta \in \mathfrak{l}$ and $x\in S^{n},$ 
\begin{equation}
\left\langle \eta ,x\right\rangle =\left\langle h^{\prime }\left( \eta
\right) ,h\left( s\right) _{x}\right\rangle .  \label{etax}
\end{equation}

Suppose now $\xi $ is a unit vector in $\mathfrak{l}$. Noting that $h\left(
t\xi \cdot s\right) =\left( th^{\prime }\left( \xi \right) \right) \cdot
h\left( s\right) $, from (\ref{floweq5}) we see $h\left( s\left( t\right)
\right) =h\left( t\xi \cdot s\right) $ satisfies the following equations 
\begin{equation}
\left\{ 
\begin{array}{c}
\frac{dh\left( s\left( t\right) \right) _{x}}{dt}=h^{\prime }\left( \xi
\right) -\hat{\xi}h\left( s\left( t\right) \right) _{x} \\ 
\frac{dh\left( s\left( t\right) \right) _{\phi }}{dt}=\hat{\xi}%
\end{array}%
\right. .  \label{hfloweq}
\end{equation}%
From (\ref{etax}), we see that $\left\langle th^{\prime }\left( \xi \right)
,h\left( s\left( t\right) \right) _{x}\right\rangle =\left\langle t\xi
,x\left( t\right) \right\rangle $, and hence $\hat{\xi}$ is unchanged from (%
\ref{floweq5}). Taking inner product of the first equation of (\ref{hfloweq}%
) with $h^{\prime }\left( \xi \right) $ yields 
\begin{equation*}
\frac{d\hat{\xi}}{dt}=\left\vert h^{\prime }\left( \xi \right) \right\vert
^{2}-\hat{\xi}^{2}.
\end{equation*}%
However, comparing with (\ref{xihateq}), we see that $\left\vert h^{\prime
}\left( \xi \right) \right\vert =1.$ Therefore, $h^{\prime }\in SO\left(
n+1\right) $, and hence from (\ref{etax}) we see that $h\left( s\left(
t\right) \right) _{x}=h^{\prime }\left( x\left( t\right) \right) .$

We now need to understand how $h$ transforms the $\phi $ component of $s$.
From (\ref{x0sol}) and (\ref{x1sol}), for $\xi \in \mathfrak{l}$, let us
denote the $x$ and $\phi $ components of $\xi \cdot s\ $by $\xi \cdot x\ $%
and $\phi +\phi _{\xi }$, respectively. Define $\tilde{h}_{x}\left( \phi
\right) =h\left( s\right) _{\phi }$ to be the transformation of the $\phi $
component, so that $h\left( s\right) =h\left( x,\phi \right) =\left(
h^{\prime }\left( x\right) ,\tilde{h}_{x}\left( \phi \right) \right) .$ In
particular, 
\begin{eqnarray*}
h\left( \xi \cdot s\right) &=&h\left( \xi \cdot x,\phi +\phi _{\xi }\right)
\\
&=&\left( h^{\prime }\left( \xi \cdot x\right) ,\tilde{h}_{\xi \cdot
x}\left( \phi +\phi _{\xi }\right) \right) .
\end{eqnarray*}%
On the other hand, from the definition of $h$, 
\begin{eqnarray*}
h\left( \xi \cdot s\right) &=&h^{\prime }\left( \xi \right) \cdot h\left(
s\right) \\
&=&h^{\prime }\left( \xi \right) \cdot \left( h^{\prime }\left( x\right) ,%
\tilde{h}_{x}\left( \phi \right) \right) \\
&=&\left( h^{\prime }\left( \xi \right) \cdot h^{\prime }\left( x\right) ,%
\tilde{h}_{x}\left( \phi \right) +\phi _{h^{\prime }\left( \xi \right)
}\right) .
\end{eqnarray*}%
We already know that $h^{\prime }\left( \xi \cdot x\right) =h^{\prime
}\left( \xi \right) \cdot h^{\prime }\left( x\right) ,$ and moreover, since $%
\left\langle \xi ,x\right\rangle =\left\langle h^{\prime }\left( \xi \right)
,h^{\prime }\left( x\right) \right\rangle $, $\phi _{h^{\prime }\left( \xi
\right) }=\phi _{\xi }$. Thus we see that 
\begin{equation}
\tilde{h}_{\xi \cdot x}\left( \phi +\phi _{\xi }\right) =\tilde{h}_{x}\left(
\phi \right) +\phi _{\xi }.  \label{htildphi}
\end{equation}

Now from (\ref{x0sol}), we see that taking $\xi =x$, $\left( tx\right) \cdot
s=\left( x,\phi +t\right) $, so $h\left( \left( tx\right) \cdot s\right)
=\left( h^{\prime }\left( x\right) ,\tilde{h}_{x}\left( \phi +t\right)
\right) .$ However, from the definition of $h$, 
\begin{eqnarray*}
h\left( \left( tx\right) \cdot s\right) &=&th^{\prime }\left( x\right) \cdot
h\left( s\right) \\
&=&\left( h^{\prime }\left( x\right) ,\tilde{h}_{x}\left( \phi \right)
+t\right) .
\end{eqnarray*}%
Hence, we see that for any $x,\xi ,\ $and $t$, $\tilde{h}_{x}\left( \phi
+t\right) =\tilde{h}_{x}\left( \phi \right) +t.$ This shows that $\tilde{h}%
_{x}\left( \phi \right) $ $=\phi +c\left( x\right) .$ From (\ref{htildphi}),
we then see that 
\begin{equation*}
\phi +\phi _{\xi }+c\left( \xi \cdot x\right) =\phi +c\left( x\right) +\phi
_{\xi },
\end{equation*}%
which shows that $c\left( x\right) $ is independent of $x.$ Hence, the
transformation of $\phi $ is just a transformation of $S^{1}$ by an action
of $U\left( 1\right) $. Overall, we then see that $\Psi \left( \mathbb{L}%
\right) \cong SO\left( n+1\right) \times U\left( 1\right) $ and $\Psi
^{\prime }\left( \mathbb{L}\right) \cong SO\left( n+1\right) .$

Consider now more generally, $\mathbb{L=}S^{m}\times N$ where $N$ is any $n$%
-dimensional parallelizable manifold. Suppose $\left\{
T_{1},...,T_{n}\right\} $ is a frame of vector fields on $N$ that satisfy
the following Lie bracket relations:%
\begin{equation}
\left[ T_{A},T_{B}\right] =D_{AB}^{\ \ \text{\ }C}T_{C},  \label{TTbrack}
\end{equation}%
where the $D_{AB}^{\ \ \text{\ }C}$ for $A,B,C=1,...,n$ are functions on $N$%
. The idea is to take a non-trivial vector field $T$ on $N$ and add the
trivial rank 1 bundle generated by it to $TS^{m}$. The resulting rank $%
\left( m+1\right) $-bundle then be trivialized in the same way as in the
previous example of $S^{m}\times S^{1}$ \cite{BruniSpheres}. Without loss of
generality, let $T=T_{1}$. As before, consider $S^{m}\subset \mathbb{R}^{m+1}
$ and for $i=1,...,n+1$, define%
\begin{equation}
f_{i}=M_{i}+x_{i}T,  \label{biT1}
\end{equation}%
where $M_{i}$ is given by (\ref{Mi}). More explicitly, the $f_{i}$ are given
by 
\begin{equation}
f_{i}=\partial _{x_{i}}-x_{i}\sum_{j=1}^{m+1}x_{j}\partial _{x_{j}}+x_{i}T,
\end{equation}%
Then, the set $\left\{ f_{1},...,f_{n+1},T_{2},...,T_{n}\right\} $ is global
frame on $S^{m}\times N.$ From (\ref{biT1}), we see that 
\begin{equation}
T=x^{i}f_{i.}  \label{T1bi}
\end{equation}%
The basis elements $\left\{ f_{1},...,f_{n+1}\right\} $ satisfy the
following bracket relation 
\begin{equation}
\left[ f_{i},f_{j}\right] =x_{i}f_{j}-x_{j}f_{i}.  \label{bbbrack}
\end{equation}%
Also, for $A,B=2,...,n$, 
\begin{equation*}
\left[ f_{i},T_{A}\right] =\left[ x_{i}T_{1},T_{A}\right] =x_{i}D_{1A}^{\ \
1}T_{1}+x_{i}D_{1A}^{\ \ \ B}T_{B}\ 
\end{equation*}

Now let $\mathfrak{l}$ be a $\left( m+n\right) $-dimensional real vector
space, with an inner product $\left\langle \cdot ,\cdot \right\rangle $ and
an orthonormal basis $\left\{ e_{i},e_{A}^{\prime }\right\} $, where $%
i=1,...,m+1\ $and $A=2,...,n$. Then, for each $s=\left( x,y\right) \in 
\mathbb{L},$ define%
\begin{equation}
\begin{array}{c}
\rho _{s}:\mathfrak{l}\longrightarrow T_{s}\mathbb{L} \\ 
e_{i}\mapsto \left. f_{i}\right\vert _{s} \\ 
e_{A}^{\prime }\mapsto \left. T_{A}\right\vert _{s}%
\end{array}%
,
\end{equation}%
Since $T_{1}=x^{i}f_{i}$, we see that $\rho _{s}^{-1}\left( T_{1}\right)
=x^{i}e_{i}.$ Hence define $e_{1}^{\prime }=x^{i}e_{i}$. As before, using (%
\ref{lbrack}) and (\ref{bbbrack}), we define the bracket on $\mathfrak{l}$
between basis elements $e_{i}$ and $e_{j}$ as: 
\begin{eqnarray}
\left[ e_{i},e_{j}\right] ^{\left( s\right) } &=&-\rho _{s}^{-1}\left(
\left. \left[ f_{i},f_{j}\right] \right\vert _{s}\right)   \notag \\
&=&e_{i}x_{j}-e_{j}x_{i}.
\end{eqnarray}%
Similarly, for $A=2,...,n$, and using (\ref{TTbrack}), we see that 
\begin{eqnarray}
\left[ e_{i},e_{A}^{\prime }\right] ^{\left( s\right) } &=&-\rho
_{s}^{-1}\left( \left. \left[ b_{i},T_{A}\right] \right\vert _{s}\right)  
\notag \\
&=&-x_{i}\rho _{s}^{-1}\left( \left[ T_{1},T_{A}\right] \right)   \notag \\
&=&-x_{i}D_{1A}^{\ \ \ \tilde{C}}f_{\tilde{C}}  \label{brackeifA}
\end{eqnarray}%
for $\tilde{C}=1,...,n$. Moreover, 
\begin{eqnarray}
\left[ e_{i},e_{1}^{\prime }\right] ^{\left( s\right) } &=&x^{j}\left[
e_{i},e_{j}\right] ^{\left( s\right) }  \notag \\
&=&e_{i}-x_{i}e_{1}^{\prime }  \label{brackeqf1} \\
\left[ e_{1}^{\prime },e_{A}^{\prime }\right] ^{\left( s\right) } &=&x^{i}%
\left[ e_{i},e_{A}^{\prime }\right] ^{\left( s\right) }  \notag \\
&=&-D_{1A}^{\ \ \ \tilde{C}}e_{\tilde{C}}^{\prime }
\end{eqnarray}%
Now, for $B=2,...,n$, 
\begin{eqnarray}
\left[ e_{A}^{\prime },e_{B}^{\prime }\right] ^{\left( s\right) } &=&-\rho
_{s}^{-1}\left( \left. \left[ T_{A},T_{B}\right] \right\vert _{s}\right)  
\notag \\
&=&-D_{AB}^{\ \ \ \tilde{C}}e_{\tilde{C}}^{\prime }.
\end{eqnarray}%
For $\tilde{A},\tilde{B},\tilde{C}=1,..,n$, we can then summarize 
\begin{equation}
\left[ e_{\tilde{A}}^{\prime },e_{\tilde{B}}^{\prime }\right] ^{\left(
s\right) }=-D_{\tilde{A}\tilde{B}}^{\ \ \ \tilde{C}}e_{\tilde{C}}^{\prime }.
\end{equation}%
Overall, we see that for $\left[ e_{i},e_{j}\right] ^{\left( s\right) }$ we
obtain the same structure constants as in the case of $S^{m}\times S^{1}$,
for $\left[ e_{\tilde{A}}^{\prime },e_{\tilde{B}}^{\prime }\right] ^{\left(
s\right) },$ the structure constants are defined by the structure constants
on the parallelizable space $N$, and only the mixed brackets (\ref{brackeifA}%
) and (\ref{brackeqf1}) are new. We can combine these both in a single
expression for $\tilde{A},\tilde{C}=1,...,n$ as 
\begin{equation}
\left[ e_{i},e_{\tilde{A}}^{\prime }\right] ^{\left( s\right) }=-x_{i}D_{1%
\tilde{A}}^{\ \ \ \tilde{C}}e_{\tilde{C}}^{\prime }+\delta _{1\tilde{A}%
}\left( e_{i}-x_{i}e_{1}^{\prime }\right) .  \label{brackeifAt}
\end{equation}

If we consider now the Jacobi identity for $\left[ \cdot ,\cdot \right]
^{\left( s\right) }$, we find from Lemma \ref{lemSnS1Jac}, that the Jacobi
identity is satisfied for $e_{i},e_{j},e_{k}.$ Similarly, for $e_{\tilde{A}%
}^{\prime },e_{\tilde{B}}^{\prime },e_{\tilde{C}}^{\prime },$ the Jacobi
identity is given by (\ref{jacid2}). Now consider a mix of $e_{i}$ and $%
e_{A}^{\prime }$: 
\begin{eqnarray*}
\left[ e_{i},\left[ e_{j},e_{A}^{\prime }\right] ^{\left( s\right) }\right]
^{\left( s\right) } &=&-\left[ e_{i},x_{j}D_{1A}^{\ \ \ \tilde{C}}e_{\tilde{C%
}}^{\prime }\right] ^{\left( s\right) } \\
&=&x_{i}x_{j}D_{1A}^{\ \ \ \tilde{C}}D_{1\tilde{C}}^{\ \ \ \tilde{B}}e_{%
\tilde{B}}^{\prime }+x_{i}x_{j}e_{1}^{\prime }D_{1A}^{\ \ \
1}e_{i}-x_{j}D_{1A}^{\ \ \ 1}e_{i} \\
\left[ e_{j},\left[ e_{A}^{\prime },e_{i}\right] ^{\left( s\right) }\right]
^{\left( s\right) } &=&-\left[ e_{j},\left[ e_{i},e_{A}^{\prime }\right]
^{\left( s\right) }\right] ^{\left( s\right) }=-x_{i}x_{j}D_{1A}^{\ \ \ 
\tilde{C}}D_{1\tilde{C}}^{\ \ \ \tilde{B}}e_{\tilde{B}}^{\prime
}-x_{i}x_{j}e_{1}^{\prime }D_{1A}^{\ \ \ 1}e_{i}+x_{i}D_{1A}^{\ \ \ 1}e_{j}
\\
\left[ e_{A}^{\prime },\left[ e_{i},e_{j}\right] ^{\left( s\right) }\right]
^{\left( s\right) } &=&\left[ e_{A}^{\prime },e_{i}x_{j}-e_{j}x_{i}\right]
^{\left( s\right) } \\
&=&-x_{i}x_{j}D_{1A}^{\ \ \ \tilde{C}}e_{\tilde{C}}^{\prime
}+x_{i}x_{j}D_{1A}^{\ \ \ \tilde{C}}e_{\tilde{C}}^{\prime }=0
\end{eqnarray*}%
Hence, 
\begin{eqnarray*}
\left[ e_{i},\left[ e_{j},e_{A}^{\prime }\right] ^{\left( s\right) }\right]
^{\left( s\right) }+\left[ e_{j},\left[ e_{A}^{\prime },e_{i}\right]
^{\left( s\right) }\right] ^{\left( s\right) }+\left[ e_{A}^{\prime },\left[
e_{i},e_{j}\right] ^{\left( s\right) }\right] ^{\left( s\right) }
&=&D_{1A}^{\ \ \ 1}\left( x_{i}e_{j}-x_{j}e_{i}\right)  \\
&=&-D_{1A}^{\ \ \ 1}\left[ e_{i},e_{j}\right] ^{\left( s\right) }
\end{eqnarray*}%
We see that whenever $D_{1A}^{\ \ \ 1}\neq 0$, the Jacobi identity is not
satisfied, and the bracket $\left[ \cdot ,\cdot \right] ^{\left( s\right) }$
does not define a Lie algebra. In particular, even if on $N$, the bracket
algebra at $y\in N$ is a Lie algebra, if for some $A$, $D_{1A}^{\ \ \
1}\left( y\right) \neq 0$, then the bracket algebra at $s\in \mathbb{L}$ is
not a Lie algebra, and from (\ref{jacid2}), $\left( \mathfrak{l},a^{\left(
s\right) }\right) $ is not a Lie triple system.

As a particular example, consider $S^{m}\times S^{n}\times S^{1}.$ We know $%
S^{n}\times S^{1}$ is parallelizable, and using the same parallelization on $%
S^{n}\times S^{1}$ as we did earlier on $S^{m}\times S^{1}$, but labelling
the basis elements on $T\left( S^{m}\times S^{1}\right) $ as $T_{A}$ for $%
A=1,...,n+1$, from (\ref{spherebrack}), we find that in this case, for $%
s=\left( x,y,\phi \right) $, with $x$ being the coordinates in $\mathbb{R}%
^{m+1}\supset S^{m}$, $y$ being the coordinates in $\mathbb{R}^{n+1}\supset
S^{n}$, and $\phi $ being the coordinate on $\phi $, 
\begin{equation*}
\left. \left[ T_{A},T_{B}\right] \right\vert _{y}=y_{A}T_{B}-y_{B}T_{A}.
\end{equation*}%
in this case, for $A\neq 1$, 
\begin{equation*}
D_{1A}^{\ \ 1}\left( y\right) =-y_{A}.
\end{equation*}%
Thus, as long as $y\neq \left( 1,0,...,0\right) ,$ this will be non-zero for
some $A$, and the algebra of $\left[ \cdot ,\cdot \right] ^{\left( s\right)
} $ for $s=\left( x,y\right) $ will not satisfy the Jacobi identity.

\appendix

\section{Appendix}

\label{secApp}\setcounter{equation}{0}

\begin{proof}[Proof of Theorem \protect\ref{thmTriple}]
To complete the proof of Theorem \ref{thmTriple}, we need to show that given 
\begin{equation}
a^{\left( s\right) }\left( \xi ,\eta ,\gamma \right) =\gamma \left(
\left\langle \xi ,\eta \right\rangle +\left\langle x,\xi \right\rangle
\left\langle x,\eta \right\rangle \right) -\eta \left( \left\langle \xi
,\gamma \right\rangle +\left\langle x,\xi \right\rangle \left\langle
x,\gamma \right\rangle \right) ,
\end{equation}%
for any $\xi ,\eta ,\gamma ,u,v\in \mathfrak{l},$ 
\begin{eqnarray}
a_{\xi ,\eta }^{\left( s\right) }\left( a^{\left( s\right) }\left(
u,v,\gamma \right) \right) &=&a^{\left( s\right) }\left( a^{\left( s\right)
}\left( u,\xi ,\eta \right) ,v,\gamma \right) \\
&&+a^{\left( s\right) }\left( u,a^{\left( s\right) }\left( v,\xi ,\eta
\right) ,\gamma \right)  \notag \\
&&+a^{\left( s\right) }\left( u,v,a^{\left( s\right) }\left( \gamma ,\xi
,\eta \right) \right) .  \notag
\end{eqnarray}%
Now, let 
\begin{equation*}
A=a^{\left( s\right) }\left( u,v,\gamma \right) =\gamma \left( \left\langle
u,v\right\rangle +\left\langle x,u\right\rangle \left\langle
x,v\right\rangle \right) -v\left( \left\langle u,\gamma \right\rangle
+\left\langle x,u\right\rangle \left\langle x,\gamma \right\rangle \right)
\end{equation*}%
, 
\begin{equation*}
a^{\left( s\right) }\left( A,\xi ,\eta \right) =\eta \left( \left\langle
A,\xi \right\rangle +\left\langle x,A\right\rangle \left\langle x,\xi
\right\rangle \right) -\xi \left( \left\langle A,\eta \right\rangle
+\left\langle x,A\right\rangle \left\langle x,\eta \right\rangle \right) 
\text{,}
\end{equation*}%
but 
\begin{eqnarray*}
\left\langle A,\xi \right\rangle &=&\left\langle \gamma ,\xi \right\rangle
\left\langle u,v\right\rangle +\left\langle \gamma ,\xi \right\rangle
\left\langle x,u\right\rangle \left\langle x,v\right\rangle -\left\langle
v,\xi \right\rangle \left\langle u,\gamma \right\rangle -\left\langle v,\xi
\right\rangle \left\langle x,u\right\rangle \left\langle x,\gamma
\right\rangle \\
\left\langle A,\eta \right\rangle &=&\left\langle \gamma ,\eta \right\rangle
\left\langle u,v\right\rangle +\left\langle \gamma ,\eta \right\rangle
\left\langle x,u\right\rangle \left\langle x,v\right\rangle -\left\langle
v,\eta \right\rangle \left\langle u,\gamma \right\rangle -\left\langle
v,\eta \right\rangle \left\langle x,u\right\rangle \left\langle x,\gamma
\right\rangle \\
\left\langle A,x\right\rangle &=&\left\langle \gamma ,x\right\rangle
\left\langle u,v\right\rangle +\left\langle \gamma ,x\right\rangle
\left\langle x,u\right\rangle \left\langle x,v\right\rangle -\left\langle
v,x\right\rangle \left\langle u,\gamma \right\rangle -\left\langle
v,x\right\rangle \left\langle x,u\right\rangle \left\langle x,\gamma
\right\rangle \\
&=&\left\langle \gamma ,x\right\rangle \left\langle u,v\right\rangle
-\left\langle v,x\right\rangle \left\langle u,\gamma \right\rangle .
\end{eqnarray*}%
Thus, 
\begin{eqnarray*}
a^{\left( s\right) }\left( a^{\left( s\right) }\left( u,v,\gamma \right)
,\xi ,\eta \right) &=&\eta \left( \left\langle \gamma ,\xi \right\rangle
\left\langle u,v\right\rangle +\left\langle \gamma ,\xi \right\rangle
\left\langle x,u\right\rangle \left\langle x,v\right\rangle -\left\langle
v,\xi \right\rangle \left\langle u,\gamma \right\rangle -\left\langle v,\xi
\right\rangle \left\langle x,u\right\rangle \left\langle x,\gamma
\right\rangle \right) \\
&&-\xi \left( \left\langle \gamma ,\eta \right\rangle \left\langle
u,v\right\rangle +\left\langle \gamma ,\eta \right\rangle \left\langle
x,u\right\rangle \left\langle x,v\right\rangle -\left\langle v,\eta
\right\rangle \left\langle u,\gamma \right\rangle -\left\langle v,\eta
\right\rangle \left\langle x,u\right\rangle \left\langle x,\gamma
\right\rangle \right) \\
&&+\left( \eta \left\langle x,\xi \right\rangle -\xi \left\langle x,\eta
\right\rangle \right) \left( \left\langle \gamma ,x\right\rangle
\left\langle u,v\right\rangle -\left\langle v,x\right\rangle \left\langle
u,\gamma \right\rangle \right)
\end{eqnarray*}%
Similarly, 
\begin{eqnarray*}
a^{\left( s\right) }\left( a^{\left( s\right) }\left( u,\xi ,\eta \right)
,v,\gamma \right) &=&\gamma \left( \left\langle \eta ,v\right\rangle
\left\langle u,\xi \right\rangle +\left\langle \eta ,v\right\rangle
\left\langle x,u\right\rangle \left\langle x,\xi \right\rangle -\left\langle
v,\xi \right\rangle \left\langle u,\eta \right\rangle -\left\langle v,\xi
\right\rangle \left\langle x,u\right\rangle \left\langle x,\eta
\right\rangle \right) \\
&&-v\left( \left\langle \gamma ,\eta \right\rangle \left\langle u,\xi
\right\rangle +\left\langle \gamma ,\eta \right\rangle \left\langle
x,u\right\rangle \left\langle x,\xi \right\rangle -\left\langle \xi ,\gamma
\right\rangle \left\langle u,\eta \right\rangle -\left\langle \xi ,\gamma
\right\rangle \left\langle x,u\right\rangle \left\langle x,\eta
\right\rangle \right) \\
&&+\left( \gamma \left\langle x,v\right\rangle -v\left\langle x,\gamma
\right\rangle \right) \left( \left\langle \eta ,x\right\rangle \left\langle
u,\xi \right\rangle -\left\langle \xi ,x\right\rangle \left\langle u,\eta
\right\rangle \right)
\end{eqnarray*}%
Now let 
\begin{eqnarray*}
B &=&a^{\left( s\right) }\left( v,\xi ,\eta \right) =\eta \left(
\left\langle v,\xi \right\rangle +\left\langle x,v\right\rangle \left\langle
x,\xi \right\rangle \right) -\xi \left( \left\langle v,\eta \right\rangle
+\left\langle x,v\right\rangle \left\langle x,\eta \right\rangle \right) \\
\left\langle B,u\right\rangle &=&\left\langle \eta ,u\right\rangle
\left\langle v,\xi \right\rangle +\left\langle \eta ,u\right\rangle
\left\langle x,v\right\rangle \left\langle x,\xi \right\rangle -\left\langle
\xi ,u\right\rangle \left\langle v,\eta \right\rangle -\left\langle \xi
,u\right\rangle \left\langle x,v\right\rangle \left\langle x,\eta
\right\rangle \\
\left\langle B,x\right\rangle &=&\left\langle \eta ,x\right\rangle
\left\langle v,\xi \right\rangle -\left\langle \xi ,x\right\rangle
\left\langle v,\eta \right\rangle
\end{eqnarray*}%
and 
\begin{eqnarray*}
a^{\left( s\right) }\left( u,a^{\left( s\right) }\left( v,\xi ,\eta \right)
,\gamma \right) &=&\gamma \left( \left\langle u,B\right\rangle +\left\langle
x,u\right\rangle \left\langle x,B\right\rangle \right) -B\left( \left\langle
u,\gamma \right\rangle +\left\langle x,u\right\rangle \left\langle x,\gamma
\right\rangle \right) \\
&=&\gamma \left( \left\langle \eta ,u\right\rangle \left\langle v,\xi
\right\rangle +\left\langle \eta ,u\right\rangle \left\langle
x,v\right\rangle \left\langle x,\xi \right\rangle -\left\langle \xi
,u\right\rangle \left\langle v,\eta \right\rangle -\left\langle \xi
,u\right\rangle \left\langle x,v\right\rangle \left\langle x,\eta
\right\rangle \right) \\
&&+\gamma \left\langle x,u\right\rangle \left( \left\langle \eta
,x\right\rangle \left\langle v,\xi \right\rangle -\left\langle \xi
,x\right\rangle \left\langle v,\eta \right\rangle \right) \\
&&-\eta \left( \left\langle v,\xi \right\rangle +\left\langle
x,v\right\rangle \left\langle x,\xi \right\rangle \right) \left(
\left\langle u,\gamma \right\rangle +\left\langle x,u\right\rangle
\left\langle x,\gamma \right\rangle \right) \\
&&+\xi \left( \left\langle v,\eta \right\rangle +\left\langle
x,v\right\rangle \left\langle x,\eta \right\rangle \right) \left(
\left\langle u,\gamma \right\rangle +\left\langle x,u\right\rangle
\left\langle x,\gamma \right\rangle \right)
\end{eqnarray*}%
Similarly, 
\begin{eqnarray*}
a^{\left( s\right) }\left( u,v,a^{\left( s\right) }\left( \gamma ,\xi ,\eta
\right) \right) &=&-a^{\left( s\right) }\left( u,a^{\left( s\right) }\left(
\gamma ,\xi ,\eta \right) ,v\right) \\
&=&-v\left( \left\langle \eta ,u\right\rangle \left\langle \gamma ,\xi
\right\rangle +\left\langle \eta ,u\right\rangle \left\langle x,\gamma
\right\rangle \left\langle x,\xi \right\rangle -\left\langle \xi
,u\right\rangle \left\langle \gamma ,\eta \right\rangle -\left\langle \xi
,u\right\rangle \left\langle x,\gamma \right\rangle \left\langle x,\eta
\right\rangle \right) \\
&&-v\left\langle x,u\right\rangle \left( \left\langle \eta ,x\right\rangle
\left\langle \gamma ,\xi \right\rangle -\left\langle \xi ,x\right\rangle
\left\langle \gamma ,\eta \right\rangle \right) \\
&&+\eta \left( \left\langle \gamma ,\xi \right\rangle +\left\langle x,\gamma
\right\rangle \left\langle x,\xi \right\rangle \right) \left( \left\langle
u,v\right\rangle +\left\langle x,u\right\rangle \left\langle
x,v\right\rangle \right) \\
&&-\xi \left( \left\langle \gamma ,\eta \right\rangle +\left\langle x,\gamma
\right\rangle \left\langle x,\eta \right\rangle \right) \left( \left\langle
u,v\right\rangle +\left\langle x,u\right\rangle \left\langle
x,v\right\rangle \right)
\end{eqnarray*}

Adding together the terms $a^{\left( s\right) }\left( a^{\left( s\right)
}\left( u,\xi ,\eta \right) ,v,\gamma \right) ,$ $a^{\left( s\right) }\left(
u,a^{\left( s\right) }\left( v,\xi ,\eta \right) ,\gamma \right) ,$ $%
a^{\left( s\right) }\left( u,v,a^{\left( s\right) }\left( \gamma ,\xi ,\eta
\right) \right) ,$ we see that the contributions that are proportional to $v$
and $\gamma $ vanish. Consider the contributions of $\xi :$%
\begin{eqnarray*}
&&\left( \left\langle v,\eta \right\rangle +\left\langle x,v\right\rangle
\left\langle x,\eta \right\rangle \right) \left( \left\langle u,\gamma
\right\rangle +\left\langle x,u\right\rangle \left\langle x,\gamma
\right\rangle \right) -\left( \left\langle \gamma ,\eta \right\rangle
+\left\langle x,\gamma \right\rangle \left\langle x,\eta \right\rangle
\right) \left( \left\langle u,v\right\rangle +\left\langle x,u\right\rangle
\left\langle x,v\right\rangle \right) \\
&=&\left\langle v,\eta \right\rangle \left\langle u,\gamma \right\rangle
+\left\langle v,\eta \right\rangle \left\langle x,u\right\rangle
\left\langle x,\gamma \right\rangle +\left\langle x,v\right\rangle
\left\langle x,\eta \right\rangle \left\langle u,\gamma \right\rangle
+\left\langle x,v\right\rangle \left\langle x,\eta \right\rangle
\left\langle x,u\right\rangle \left\langle x,\gamma \right\rangle \\
&&-\left\langle \gamma ,\eta \right\rangle \left\langle u,v\right\rangle
-\left\langle \gamma ,\eta \right\rangle \left\langle x,u\right\rangle
\left\langle x,v\right\rangle -\left\langle x,\gamma \right\rangle
\left\langle x,\eta \right\rangle \left\langle u,v\right\rangle
-\left\langle x,\gamma \right\rangle \left\langle x,\eta \right\rangle
\left\langle x,u\right\rangle \left\langle x,v\right\rangle \\
&=&\left\langle v,\eta \right\rangle \left\langle u,\gamma \right\rangle
-\left\langle \gamma ,\eta \right\rangle \left\langle u,v\right\rangle
+\left\langle v,\eta \right\rangle \left\langle x,u\right\rangle
\left\langle x,\gamma \right\rangle +\left\langle x,v\right\rangle
\left\langle x,\eta \right\rangle \left\langle u,\gamma \right\rangle \\
&&-\left\langle \gamma ,\eta \right\rangle \left\langle x,u\right\rangle
\left\langle x,v\right\rangle -\left\langle x,\gamma \right\rangle
\left\langle x,\eta \right\rangle \left\langle u,v\right\rangle
\end{eqnarray*}%
This is precisely the coefficient of $\xi $ in $a^{\left( s\right) }\left(
a^{\left( s\right) }\left( u,v,\gamma \right) ,\xi ,\eta \right) .$ The
expression is skew-symmetric in $\xi $ and $\eta ,$ so the same conclusion
follows.
\end{proof}

\bibliographystyle{s:/tex/BibTeX/bst/habbrv}
\bibliography{s:/tex/input/refs2}

\end{document}